\documentclass[12pt]{amsart}
\usepackage[scale=.86]{geometry}
\usepackage{geometry} 
\usepackage{graphicx}
\usepackage{bbm}
\usepackage{mathrsfs}
\theoremstyle{plain}

\usepackage{CJK}
\usepackage{amsfonts,amssymb,amsmath,mathrsfs,multirow,booktabs}

\usepackage{color,latexsym,amsfonts}
 \newtheorem{theorem}{Theorem}[section]
 \newtheorem{lemma}[theorem]{Lemma}
 \newtheorem{corollary}[theorem]{Corollary}
 \newtheorem{proposition}[theorem]{Proposition}
 \newtheorem{example}[theorem]{Example}
 
\theoremstyle{remark}
\newtheorem{remark}[theorem]{Remark}
 \newtheorem{condition}[theorem]{Condition}

 \def\beqlb{\begin{eqnarray}}\def\eeqlb{\end{eqnarray}}
 \def\beqnn{\begin{eqnarray*}}\def\eeqnn{\end{eqnarray*}}

 \def\<{\langle}\def\>{\rangle}

 \def\eqref#1{{\rm(\ref{#1})}}

 \def\qed{\hfill$\Box$\medskip}

\def\e{{\mbox{\rm e}}}
\def\<{\left<}\def\>{\right>}

\newcommand{\dd}{\mathrm{d}}

\newfam\msbmfam\font\tenmsbm=msbm10\textfont
\msbmfam=\tenmsbm\font\sevenmsbm=msbm7
\scriptfont\msbmfam=\sevenmsbm

\def\<{\left<}\def\>{\right>}

\def\({\left(}\def\){\right)}

\title[Ancestral lineages in  subcritical continuous-state branching populations]{ \bf Asymptotic behaviour of ancestral lineages in subcritical continuous-state branching populations}
\usepackage[pagebackref,colorlinks,citecolor=Plum,urlcolor=Periwinkle,linkcolor=DarkOrchid]{hyperref}
\newcommand{\ddr}{\mathrm{d}}
\usepackage[dvipsnames]{xcolor}

\date{December 7, 2020} 

\begin{document}
\maketitle

\centerline{\large  Cl\'ement Foucart \footnote{Universit\'e  Sorbonne Paris Nord and Paris 8, Laboratoire Analyse, G\'eom\'etrie $\&$ Applications, UMR 7539. Institut Galil\'ee, 99 avenue J.B. Cl\'ement, 93430 Villetaneuse, France; foucart@math.univ-paris13.fr} and Martin  M\"ohle \footnote{
Eberhard Karls Universit\"at Tübingen, Auf der Morgenstelle 10, 72076 T\"ubingen, Germany; martin.moehle@uni-tuebingen.de.}}
%\date{December 7, 2020}
\begin{abstract}
Consider the population model with infinite size associated to subcritical continuous-state branching processes (CSBP). Individuals reproduce independently according to the same subcritical offspring distribution. We study the long-term behaviour of the ancestral lineages as time goes to the past and show that the flow of ancestral lineages, properly renormalized, converges almost surely to the inverse of a drift-free subordinator whose Laplace exponent is explicit in terms of the branching mechanism. We provide an interpretation in terms of the genealogy of the population. In particular,  we show that the inverse subordinator is partitioning the current population  into ancestral families with distinct common ancestors.  
When Grey’s condition is satisfied, the population comes from a discrete set of ancestors and the  ancestral families are i.i.d and distributed according to the quasi-stationary distribution of the CSBP conditioned on non-extinction. When Grey’s condition is not satisfied, the population comes from a continuum of ancestors which is described as the set of increase points $\mathscr{S}$ of the limiting inverse subordinator. The Hausdorff dimension of $\mathscr{S}$ is given. The proof is based on a general result for stochastically monotone processes of independent interest, which relates $\theta$-invariant measures and $\theta$-invariant functions for a process and its Siegmund dual. 
\end{abstract}
\section{Introduction} 
Continuous-state branching processes (CSBPs) are positive Markov processes satisfying the branching property. They arise as scaling limits of Galton-Watson processes and form a  fundamental class of random population models. Their longterm behavior has received a great deal of attention since the seventies. We refer to the early works of Bingham \cite{MR0410961} and Grey \cite{MR0408016}.  In the seminal work \cite{MR1771663}, Bertoin and Le Gall showed how to encode a complete genealogy of a random branching population by considering a flow of subordinators $(X_{s,t}(x), s\leq t, x\geq 0).$

In this representation the population has an infinite size at all times and all individuals have arbitrarily old ancestors. More precisely, individuals at time $s$ with descendants at time $t$ are the locations of the jumps of the subordinator $x\mapsto X_{s,t}(x)$ and the descendants at time $t$ of the individuals in the population at time $s$ are represented by the jump intervals. We shall provide more background on the flow of subordinators in the sequel.

Most works on CSBPs focus on their long-term behaviour forward in time. We refer to Bertoin et al. \cite{MR2455180}, Duquesne and Labb\'{e} \cite{MR3164759}, Labb\'{e} \cite{labbe2012genealogy} and Foucart and Ma \cite{foucart2019} for studies in the framework of the flow of subordinators.  
However, the representation of the population model through $(X_{s,t}(x),s\leq t, x\geq 0)$ allows one to follow the \textit{ancestral lineages backward in time}. In this article, we are interested in the backward genealogy of the continuous population and how it behaves on the long-term. To the best of our knowledge, fewer works on CSBPs have been done in this direction. We refer however to Labb\'{e} \cite{labbe2012genealogy}, Lambert \cite{MR2014270}, Lambert and Popovic \cite{MR3059232} and Foucart et al. \cite{FoMaMa2019}. The latter work initiates the study of the inverse flow $(\hat{X}_{s,t}(x), s\leq t, x\geq 0)$
defined for $s\leq t$ and $x\in [0,\infty]$, as
\begin{equation}\label{dualflowdef} \hat{X}_{s,t}(x):=\inf \{y\geq 0:\ X_{-t,-s}(y)> x\}.
\end{equation}
This random variable represents the ancestor at time $-t$ of the individual $x$ in the population at time $-s$. From now on consider an arrow of time pointing to the past, and call $\hat{X}_t(x):=\hat{X}_{0,t}(x)$, the ancestor at time $t\geq 0$ (backwards) of the individual $x$ of the population at time $0$. The two-parameter flow $(\hat{X}_t(x),t\geq 0, x\geq 0)$ is therefore representing the ancestral lineages of  the individuals in the current population. 
We call $(\hat{X}_t(x),t\geq 0)$ the \textit{ancestral lineage process}. This is a Feller process with no positive jumps, namely  $\mathbb{P}\big(\sup_{t>0}(\hat{X}_{t}(x)-\hat{X}_{t-}(x))>0\big)=0$. Moreover, for any $x\neq y$, whenever $(\hat{X}_t(x),t\geq 0)$ and $(\hat{X}_t(y),t\geq 0)$ cross, they coalesce and such a coalescence represents the occurrence in the past of a common ancestor of the individuals $x$ and $y$. We stress that coalescence can be multiple in the sense that more than two lineages can coalesce at the same time. We refer to \cite{FoMaMa2019} for a study of the coalescent processes embedded in the flow  $(\hat{X}_t(x),t\geq 0, x\geq 0)$. 
%
%For any fixed $x\in [0,\infty]$, the Markov process $(\hat{X}_t(x),t\geq 0)$ corresponds to the so-called Siegmund dual of the branching process $(X_t(x),t\geq 0)$. It satisfies the duality relation
%\begin{equation}\label{duality} \mathbb{P}(\hat{X}_t(x)\geq y)=\mathbb{P}(x\leq X_t(y)).\end{equation}
%%To the best of our knowledge, only a few attention has been paid to the dual process $(\hat{X}_t(x),t\geq 0)$.  
%It has been previously studied in discrete-state space, but without interpretation in terms of ancestral lineages, by Pakes et al. \cite{MR2409319} and Pakes \cite{pakes2017}. 
%In the continuous-state space, a characterization of its semigroup, its generator and transience/recurrence properties  have been obtained in \cite{FoMaMa2019}. 
%It is also worth noticing that such a process naturally occurs in the study of the number of blocks in a Bolthausen-Sznitman coalescent see M\"ohle \cite{MR3333734}, Kukla and M\"ohle \cite{MR3758343}. 
%Recently, it has been remarked in Foucart et al. \cite{} that the Siegmund dual of a branching process has a direct interpretation in terms of genealogy   

We focus on subcritical CSBPs. 
%The main purpose is to study the ancestral lineages of the individuals standing at some initial time,  as time goes to $\infty$ into the past. 
In such a setting, it is known, see  \cite[Proposition 2.8]{FoMaMa2019}, that for all $x\in (0,\infty)$, the Markov process $(\hat{X}_t(x),t\geq 0)$ is transient. We shall find an almost sure renormalisation of $(\hat{X}_t(x),t\geq 0)$ as $t$ goes to $\infty$ and study the limiting process in the variable $x$. This leads to an almost sure description of the long-term behaviour of the ancestral lineages  for all subcritical CSBPs, including those not satisfying Grey's condition, see Section \ref{background} and Theorem \ref{mainthm}. 
%It completes some preliminary results established in \cite[Proposition 2.1 and Theorem 4.24]{FoMaMa2019} when Grey's condition holds (namely when families get extinct in finite time almost surely).  

%Further background on continuous-state branching processes (CSBPs for short) and the representation in terms of flow of subordinators, is provided in Section \ref{background}. 

The paper is organised as follows. Further background on CSBPs and their representation in terms of flow of subordinators are provided in Section \ref{background}. 
%We give in Section \ref{background} backgrounds on continuous-state branching processes and briefly explain the flow representation of the associated continuous population. 
Fundamental properties of the ancestral lineage process $(\hat{X}_t(x),t\geq 0)$, such as its Siegmund duality relation with the CSBP  $(X_t(x),t\geq 0)$ and the representation of its semigroup, are also recalled. Our main results are stated in Section \ref{results} and proven in Section \ref{proofs}. The proof is based on a general result, established in Theorem \ref{invariantfunction}, for stochastically monotone Markov processes by showing how to link (infinite) $\theta$-invariant measures of a process $(X_t,t\geq 0)$ with (increasing) $\theta$-invariant functions of its Siegmund dual process $(\hat{X}_t,t\geq 0)$. We apply this result in the setting of CSBPs. 
%First, we show an almost sure convergence for the process $(\hat{X}_t(x),t\geq 0)$ started from a fixed value $x$, see Lemma \ref{asconvlambda}. In a second time, we study the associated limiting process in $x$ and show that it satisfies the properties stated in Theorem \ref{mainthm}. We will see from the proof of Theorem \ref{mainthm} that Proposition \ref{ancestralpartition} holds true.
%Proposition \ref{dimension} will be a consequence of Theorem \ref{mainthm}.
\section{Background on CSBPs and the flow of subordinators}\label{background}
We first recall basic definitions and properties of CSBPs and their representation in terms of flows. These processes are continuous time and continuous space analogue of Galton-Watson Markov chains. They have been introduced by Lamperti \cite{MR0208685} and Ji{\v{r}}ina \cite{MR0101554}. CSBPs are positive Markov processes satisfying the branching property: for any $x,y\geq 0$ and fixed time $t\geq 0$,
\begin{equation}\label{branching} X_t(x+y)=X'_t(x)+X''_t(y),\end{equation}
where $(X_t(x+y),t\geq 0)$ is a CSBP started from $x+y$, and $(X'_t(x),t\geq 0)$ and $(X''_t(y),t\geq 0)$ are two independent copies of the process started respectively from $x$ and $y$. We refer the reader to \cite[Chapter 3]{MR2760602} for an introduction to CSBPs. Denote by $\mathcal{L}$ the generator of  $(X_t(x),t\geq 0)$. For any $q\geq 0$, set $e_q(x):=e^{-qx}$ for any $x\geq 0$. The operator $\mathcal{L}$ acts on the exponential functions as follows. For any $q\geq 0$, and $x\in [0,\infty)$
\begin{equation}\label{generatoronexponential}\mathcal{L}e_q(x)=\Psi(q)xe_q(x),
\end{equation}
where $\Psi$ is a L\'{e}vy-Khintchine function and is called the branching mechanism. We refer e.g. to Silverstein \cite{MR0226734}.  The linear span of exponential functions $A:=\mathrm{Span}(\{e_q,q\in [0,\infty)\})$ is a core for generator $\mathcal{L}$.

We shall merely be interested in subcritical CSBPs for which $\Psi$  is of the form
\begin{equation}\label{branchingmechanism} \Psi(u)=\frac{\sigma^2}{2}u^2+\gamma u +\int_{0}^{\infty}(e^{-ux}-1+ux)\pi(\ddr x) \text{ for all } u\geq 0,
\end{equation}
%with $\sigma\geq 0$, $\gamma=\Psi(0+)>0$ and $\pi$ such that $\int_0^{\infty}(x\wedge x^{2})\pi(\ddr x)<\infty$.  We discard in the sequel the trivial deterministic case for which the branching mechanism $\Psi$ is linear. Namely, 
where $\gamma=\Psi'(0+)>0$, $\sigma\geq 0$, and $\pi$ is a L\'{e}vy measure, i.e. a Borel measure such that $\int_{0}^{\infty}(x\wedge x^2)\pi(\ddr x)<\infty$. We assume that either $\pi \neq 0$ or $\sigma>0$, so that $\Psi$ is not linear.
%and to Bertoin and Le Gall \cite{MR1771663} and \cite[Section 1]{FoMaMa2019} for the definition of the associated flow of subordinators.
%of CSBPs $(X_{s,t}(x),t\geq s, x\geq 0)$ 
%with mechanism $\Psi$ 
The semigroup of $(X_t(x),t\geq 0)$ satisfies for any $\lambda\in (0,\infty)$, $t\geq 0$ and $x\in [0,\infty)$
\begin{equation}\label{semigroupX}
\mathbb{E}[e^{-\lambda X_t(x)}]=e^{-xv_t(\lambda)} %\text{ with }
%v_t(\lambda) \text{ solution to the equation } \int_{v_t(\lambda)}^{\lambda}\frac{\ddr u}{\Psi(u)}=t.
\end{equation}
with $t\mapsto v_t(\lambda)$ for any $\lambda\in (0,\infty)$, as the solution to the integral equation 
\begin{equation}\label{v}\int_{v_t(\lambda)}^{\lambda}\frac{\ddr u}{\Psi(u)}=t.
\end{equation}
Note that $t\mapsto v_t(\lambda)$ solves $\frac{\ddr }{\ddr t}v_t(\lambda)=-\Psi(v_t(\lambda))$ with $v_0(\lambda)=\lambda$.
\noindent As a first consequence of \eqref{semigroupX}, the process $(X_t(x),t\geq 0)$ gets extinct at time $t$ with probability $e^{-xv_t(\infty)}$ where $v_t(\infty):=\underset{\lambda \rightarrow \infty}{\lim} v_t(\lambda)\in (0,\infty]$. The latter is finite if and only if $\Psi$ satisfies Grey's condition 
\begin{equation}\label{Greycondition}
\int^{\infty}\frac{\dd u}{\Psi(u)}<\infty.
\end{equation}

Lambert \cite{MR2299923} and Li \cite{MR1727226} have studied the subcritical process $(X_t(x),t\geq 0)$ conditioned on non-extinction and established the following weak convergence when \eqref{Greycondition} holds: \[\mathbb{P}(X_t(x)\in \cdot \ |X_t(x)>0)\underset{t\rightarrow \infty}{\longrightarrow} \nu_{\infty}(\cdot),\]
where $\nu_{\infty}$, the so-called quasi-stationary distribution of the CSBP, has Laplace transform
\begin{equation}\label{qsd}\int_{0}^{\infty}e^{-q x}\nu_{\infty}(\ddr x)=1-\kappa_{\infty}(q):=1-e^{-\Psi'(0+)\int_{q}^{\infty}\frac{\ddr u}{\Psi(u)}}, \ q\geq 0.
\end{equation}

The branching property \eqref{branching} can be translated in terms of independence and stationarity of the increments of the process $(X_t(x),x\geq 0)$ for any fixed time $t$. The latter is therefore a subordinator and according to \eqref{semigroupX}, its Laplace exponent is $\lambda\mapsto v_t(\lambda)$.
%, which takes thus the form
%\[v_t(\lambda)=d_t\lambda+\int_{0}^{\infty}(1-e^{-qx})\ell_t(\ddr x),
%with d_t=e^{-Dt}.
Starting from this observation, Bertoin and Le Gall in \cite{MR1771663} showed that a complete population model can be associated to CSBPs through a flow of subordinators $(X_{s,t}(x), s \leq t, x \geq 0)$. 
%\begin{definition}
%\label{flowdef2}

More precisely, the collection of processes $(X_{s,t}(x), s \leq t, x \geq 0)$ is satisfying the following properties:
\begin{enumerate}
  \item For every $s \leq t$, $x \mapsto X_{s,t}(x)$ is a càdlàg subordinator with Laplace exponent $\lambda\mapsto v_{t-s}(\lambda)$.
  \item For every $t \in \mathbb{R}$, $(X_{r,s}, r \leq s \leq t)$ and $(X_{r,s}, t\leq r \leq s)$ are independent.
  \item For every $r \leq s \leq t$, $X_{r,t} = X_{s,t} \circ X_{r,s}$.
 % \item For every $s \in \mathbb{R}$ and $x \geq 0$, we have $X_{s,s}(x) = x = \lim_{t \to s} X_{s,t}(x)$ in probability.
\end{enumerate}
%\end{definition}
The two-parameter flow $(X_{t}(x),x\geq 0, t\geq 0):=(X_{0,t}(x),x\geq 0, t\geq 0)$ is a flow of CSBPs with branching mechanism $\Psi$, each starting from an initial population of arbitrarily large size $x$. The three-parameter flow above provides a complete genealogy of the underlying (infinite) population: let $y\in (0,\infty)$, if $X_{s,t}(y-)<X_{s,t}(y)$, then the individual $y$ at time $s$ has descendants at time $t$ and those are represented by the interval $(X_{s,t}(y-),X_{s,t}(y)]$; see Figure \ref{flowrepre}.

\begin{figure}[!ht]
\centering \noindent
\includegraphics[height=.22 \textheight]{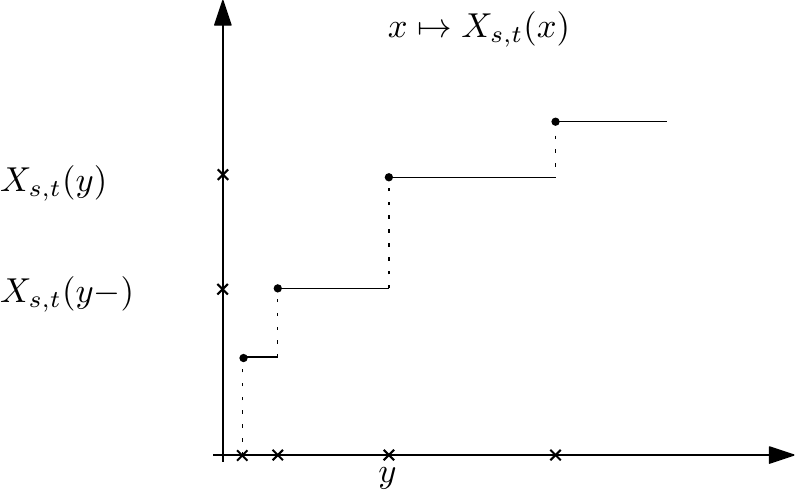}
\caption{Symbolic representation of the genealogy forward in time}
\label{flowrepre}
\end{figure}

We work now on the probability space on which the flow $(X_{s,t}(x),  s\leq t, x\geq 0)$ is defined. Recall the definition \eqref{dualflowdef} of the inverse flow and set 
$$(\hat{X}_{t}(x),t\geq 0):=(\hat{X}_{0,t}(x),t\geq 0).$$
%For any fixed $x\in [0,\infty]$, the Markov process $(\hat{X}_t(x),t\geq 0)$ corresponds to the so-called \textit{Siegmund dual} of the branching process $(X_t(x),t\geq 0)$. The following duality relationship is fundamental in our study: for any $t$ and $x,y$
%\begin{equation}\label{dualitycsbp}\mathbb{P}\big(\hat{X}_t(x)<y\big)=\mathbb{P}\big( x<X_t(y)\big).
%\end{equation}
%It satisfies the duality relation
%\begin{equation}\label{duality} \mathbb{P}(\hat{X}_t(x)\geq y)=\mathbb{P}(X_t(y)\leq x), \text{ for any }x,y\in [0,\infty).\end{equation}  
We summarize here fundamental properties of the ancestral lineage process $(\hat{X}_t,t\geq 0)$, see \cite[Section 3]{FoMaMa2019}. We stress that here, apart if this is explicitely mentioned, the branching mechanism $\Psi$ is  general  and not assumed to be subcritical. The process $(\hat{X}_t(x),t\geq 0)$, started from $x>0$, is a non-explosive càdlàg Feller process with no positive jumps.
%; $\mathbb{P}\big(\sup_{t>0}(\hat{X}_{t}(x)-\hat{X}_{t-}(x))>0\big)=0$. 
Property (3) of the flow of subordinators $(X_{s,t}(x),s\leq t, x\geq 0)$ entails that for all $r\leq s \leq t$, 
\begin{center}
$\hat{X}_{r,t}=\hat{X}_{s,t}\circ \hat{X}_{r,s}$ a.s. 
\end{center}
This entails that the flow of processes $\big(\hat{X}_{t}(x),t\geq 0,x\in (0,\infty)\big)$
%This process represents the ancestral lineage of the individual $x$ alive at time $0$ when time runs in the past. 
is \textit{coalescing}, in the sense that, when two ancestral lineages $(\hat{X}_t(x),t\geq 0)$ and $(\hat{X}_t(y),t\geq 0)$ cross, they merge. Such a coalescence represents the occurrence in the past of a common ancestor of the individuals $x$ and $y$. 
We refer the  reader to \cite[Section 5.2]{FoMaMa2019} for a study of coalescent processes  embedded  in the flow $(\hat{X}_{s,t}, s\leq t)$ . 

The processes $X$ and $\hat{X}$ are linked through the following duality relations, see \cite[Lemma 3.3-(1) and Equation (3.5), Section 3]{FoMaMa2019}. For any $t\geq 0$ and $x,y\in (0,\infty)$
\begin{equation}\label{dualitycsbp0}
\{\hat{X}_t(x)<y\}=\{x<X_{-t,0}(y)\} \text{ almost surely.}
\end{equation}
In particular, since for any $t\geq 0$, $X_{-t,0}(y)$ has the same law as $X_{t}(y)$,
%\[
\begin{equation}\label{dualitycsbp}
\mathbb{P}\big(\hat{X}_t(x)<y\big)=\mathbb{P}\big( x<X_t(y)\big).
\end{equation}
%\]
The following other duality relations will be more convenient to work with. It also allows one to identify the Markov process $(\hat{X}_t(x),t\geq 0)$ as the \textit{Siegmund dual} of $(X_t(x),t\geq 0)$, see the forthcoming Section \ref{generalresult}.
\begin{lemma}\label{dualitylemma}
For any $t\geq 0$ and $x,y\in (0,\infty)$
\begin{equation}\label{dualitycsbp3}
\{\hat{X}_t(x)>y\}=\{x>X_{-t,0}(y)\} \text{ almost surely,}
\end{equation}
and 
\begin{equation}\label{dualitycsbp2} \mathbb{P}\big(\hat{X}_t(x)>y\big)=\mathbb{P}\big( x>X_t(y)\big).
\end{equation}
\end{lemma}
\begin{proof} Let $t\geq 0$ and $x,y\in (0,\infty)$. We establish that $\{\hat{X}_t(x)\leq y\}=\{x\leq X_t(y)\}$ almost surely. Note that this is equivalent to \eqref{dualitycsbp3}.  According to \eqref{dualitycsbp0}, 
$\{\hat{X}_t(x)< y\}=\{x<X_t(y)\}$ almost surely, therefore we only need to focus on the events $\{\hat{X}_t(x)=y\}$ and $\{X_t(y)=x\}$. Recall the definition of $\hat{X}_t(x)$ in \eqref{dualflowdef}, 
%\[
$$\{\hat{X}_t(x)=y\}=\{X_{-t,0}(y-)\leq x<X_{-t,0}(y)\} \cup \{X_{-t,0}(y-)=X_{-t,0}(y)=x\}.$$
%\]
Since $X_{-t,0}$ is a subordinator, it has no almost sure fixed discontinuities and the event $\{X_{-t,0}(y-)<X_{-t,0}(y)\}$ for fixed $y$ has probability $0$. Thus
\[\{\hat{X}_t(x)=y)\}=\{X_{-t,0}(y)=x)\} \text{ a.s,}\]
and \eqref{dualitycsbp3} is established. Since $X_{-t,0}(y)$ has the same law as $X_{t}(y)$, 
\[\mathbb{P}(\hat{X}_t(x)=y)=\mathbb{P}(X_{-t,0}(y)=x)=\mathbb{P}(X_{t}(y)=x),\]
and the identity \eqref{dualitycsbp2} holds.
%where the latter equality follows from the fact that . Finally, \eqref{dualitycsbp2} is established.\qed
\end{proof}
%It satisfies the duality relation
%\begin{equation}\label{duality} \mathbb{P}(\hat{X}_t(x)\geq y)=\mathbb{P}(X_t(y)\leq x), \text{ for any }x,y\in [0,\infty).\end{equation}  
The next theorem characterizes the semigroup of $(\hat{X}_t,t\geq 0)$.
\begin{theorem}[Theorem 3.5, Proposition 3.6 in  \cite{FoMaMa2019}]
For any continuous function $f$ defined on $(0,\infty)$ and any $q>0$,
\begin{equation}\label{semigroupXhat} \mathbb{E}[f(\hat{X}_t(\mathbbm{e}_q))]=\mathbb{E}[f(\mathbbm{e}_{v_t(q)})],
\end{equation}
where for any $\lambda \in (0,\infty)$, $\mathbbm{e}_{\lambda}$ is an exponential random variable with parameter $\lambda$.
  
The process admits an entrance boundary at $0+$ if and only if \eqref{Greycondition} is satisfied. 
\end{theorem}
When \eqref{Greycondition} is satisfied, we denote by $(\hat{X}_t(0+),t\geq 0)$ the process started from $0$, defined at any time, as $\hat{X}_t(0+):=\underset{x \downarrow 0}{\lim}\  \hat{X}_t(x)$ a.s. This corresponds to the first individual at time $t$ with descendants at time $0$.

We complete these fundamental results by showing that $\hat{X}$ satisfies some properties of regularity. For $x,y\in (0,\infty)$, set \begin{equation}\label{firstpassagetime}\hat{T}_y:=\inf\{t>0: \hat{X}_t(x)>y\}=\inf\{t>0: \hat{X}_t(x)=y\}.
\end{equation}
We shall sometimes write $\hat{T}_y=\hat{T}^x_y$ to emphasize on the initial state of the process.
\begin{lemma}[Regularity] If $-\Psi$ is not the Laplace exponent of a subordinator, then the process is regular on $(0,\infty)$, namely for any $x<y$, \[\mathbb{P}_x(\hat{T}_y<\infty)>0.\]
\end{lemma}
\begin{proof}
Let $x,y\in (0,\infty)$, for any $t>0$, $\mathbb{P}_x(\hat{T}_y<t)\geq \mathbb{P}\big(\hat{X}_t(x)>y\big)=\mathbb{P}\big(x>X_t(y)\big).$
By assumption, $-\Psi$  is not the Laplace exponent of a subordinator, this ensures that the CSBP is not almost surely non-decreasing and that the event $\{X_t(y)\underset{t\rightarrow \infty}{\longrightarrow} 0\}$ has positive probability. Hence, 
$\mathbb{P}_x(\hat{T}_y<\infty)\geq \underset{t\rightarrow \infty}{\lim}\mathbb{P}\big(x>X_t(y)\big)\geq \mathbb{P}(X_t(y)\underset{t\rightarrow \infty}{\longrightarrow} 0)>0.$\qed
\end{proof}
In the subcritical case, for which $\Psi'(0+)=\gamma>0$, one has $\underset{t\rightarrow \infty}{\lim} v_t(\lambda)=0$ and all families forward in time are getting extinct. As mentioned in the introduction, in this case the ancestral lineages are transient. 

\begin{proposition}[Proposition 3.8 in \cite{FoMaMa2019}]\label{transience} Assume $\Psi'(0+)>0$. For any $x\in (0,\infty)$, the ancestral lineage process $(\hat{X}_t(x), t\geq 0)$ is transient, i.e
%\[
$\hat{X}_t(x)\underset{t\rightarrow \infty}{\longrightarrow} \infty \text{ a.s.}$
%\]
\end{proposition}
The main aim of this article is to obtain an almost sure renormalisation of the inverse flow and to interpret it in terms of the genealogy. 
%We shall see that the speed of escape towards $\infty$ varies in function of the individual $x\in (0,\infty)$. 
Siegmund dual processes of discrete branching Markov processes have been studied by Pakes et al. \cite{MR2409319} and Pakes \cite{pakes2017}.  A striking difference for continuous-state space processes is that when Grey's condition \eqref{Greycondition} does not hold, the latter are persistent in the sense that although subcritical, they are not getting absorbed at $0$, but are decreasing towards $0$ while keeping positive mass at all times.  
%We shall follow a different route than Pakes \cite{pakes2017} in order to take this phenomenon into account. 
%We also shed light on an almost sure interpretation in terms of genealogy of the Siegmund dual. 
%Moreover, 

%We will adapt the proof of \cite[Theorem 10]{pakes2017} to establish the almost sure convergence for a fixed initial value $x$. The method follows closely the work of Barbour \cite{MR0375487},
\section{Results}\label{results}
%Our main result is the following.
%Recall $t\mapsto v_t(\lambda)$ defined in \eqref{v}, and note  that for any $\lambda>0$, $v_t(\lambda)\underset{t\rightarrow \infty}{\longrightarrow} 0$. 
%almost sure renormalisation.
\noindent Assume $\Psi'(0+)>0$. For any $\lambda \in (0,\infty)$, define the map
\begin{equation}\label{laplaceexponent} \kappa_\lambda: \theta \mapsto e^{-\Psi'(0+)\int_{\theta}^{\lambda}\frac{\ddr u}{\Psi(u)}}. 
\end{equation}
We shall see that the function $\kappa_\lambda$ is the Laplace exponent of a drift-free subordinator. We denote its L\'{e}vy measure by $\nu_{\lambda}$. The latter is finite if and only if $\int^{\infty}\frac{\ddr u}{\Psi(u)}<\infty$ (Grey's condition). In this case, the function $\kappa_\infty$ defined in \eqref{qsd},
% as
%$\kappa_\infty:\theta \mapsto e^{-\Psi'(0+)\int_{\theta}^{\infty}\frac{\ddr u}{\Psi(u)}}$
is the Laplace exponent of compound Poisson process with jump law $\nu_{\infty}$, the quasi-stationary distribution of the $\Psi$-CSBP conditioned on non-extinction.

Our main result is the following. 
\begin{theorem}\label{mainthm} Assume $\Psi'(0+)>0$.  Fix  $\lambda \in (0,\infty)$. Then, almost surely
\[v_t(\lambda)\hat{X}_t(x)\underset{t\rightarrow \infty}{\longrightarrow} \hat{W}^{\lambda}(x), \text{ for all } x\notin J^{\lambda}:=\{x>0: \hat{W}^{\lambda}(x+)>\hat{W}^{\lambda}(x)\},
\]
where $(\hat{W}^{\lambda}(x),x\geq 0)$ is a process with non-decreasing left-continuous sample paths and its right-continuous inverse process $(W^{\lambda}(y),y\geq 0)$, defined for any $y\geq 0$, by
%\[
%\begin{center}
$W^{\lambda}(y):=\inf\{x\geq 0: \hat{W}^{\lambda}(x)> y\}$, 
%\end{center}
%,\]
is a drift-free  subordinator with Laplace exponent $\kappa_\lambda$.
% $$\kappa_\lambda: \theta \mapsto e^{-\Psi'(0+)\int_{\theta}^{\lambda}\frac{\ddr u}{\Psi(u)}}=:\int_{0}^{\infty}(1-e^{-\theta x})\nu_{\lambda}(\ddr x)$$
%where $\nu_{\lambda}$ denotes its Lévy measure. 
\begin{itemize}
\item[i)] If $\int^{\infty}\frac{\ddr u}{\Psi(u)}=\infty$, then for any $\lambda \in (0,\infty)$ the process $(\hat{W}^{\lambda}(x),x\geq 0)$ has  continuous sample paths almost surely.
% if and only if $\int^{\infty}\frac{\ddr u}{\Psi(u)}=\infty$. 
%If $\int^{\infty}\frac{\ddr u}{\Psi(u)}<\infty$, then $W^{\infty}$ has discrete jumps Lévy measure the quasi-stationary distribution of the CSBP.
%\begin{itemize}
%\item[(i)] 
\item[ii)] If $\int^{\infty}\frac{\ddr u}{\Psi(u)}<\infty$, then for any $\lambda \in (0,\infty)$ the process $(\hat{W}^{\lambda}(x),x\geq 0)$ has piecewise constant sample paths almost surely, and 
\[v_t(\infty)\hat{X}_t(x) \underset{t\rightarrow \infty}{\longrightarrow} \hat{W}^{\infty}(x) \text{ for all }x\notin J^{\infty} \text{ almost surely},\] where $(\hat{W}^{\infty}(x),x\geq 0)$ is the inverse of a compound Poisson process with Laplace exponent $\kappa_\infty$.
\end{itemize}

%\item[(ii)] If $\int^{\infty}\frac{\ddr u}{\Psi(u)}=\infty$ then for any $\lambda \in (0,\infty)$, the Lévy measure of the associated subordinator $W^{\lambda}$ is infinite and the process $(\hat{W}^{\lambda}(x),x\geq 0)$ has continuous paths.
%\end{itemize}
\end{theorem}

The next observation ensures that the choice of the parameter $\lambda$ is arbitrary. A change in $\lambda$ only affects the limit by a multiplicative factor.
\begin{lemma}\label{lambda} For any $\lambda'\neq \lambda \in (0,\infty)$ and $x\in (0,\infty)$ $$\hat{W}^{\lambda'}(x)=c_{\lambda',\lambda}\hat{W}^{\lambda}(x)\text{ almost surely, with } c_{\lambda',\lambda}=e^{\Psi'(0+)\int_{\lambda}^{\lambda'}\frac{\ddr u}{\Psi(u)}}.$$
\end{lemma}
Recall $\pi$ the L\'{e}vy measure in the L\'{e}vy-Khintchine form  \eqref{branchingmechanism} of $\Psi$.  The following corollary shows that the ancestral lineage process $(\hat{X}_t(x),t\geq 0)$ has an exponential growth when the measure $\pi$  satisfies an $L\log L$ condition.
\begin{corollary}\label{LlogL} For any $\lambda>0$, $v_t(\lambda)\underset{t\rightarrow \infty}{\sim} c_{\lambda}e^{-\Psi'(0+)t}$  for some constant $c_{\lambda}>0$ if and only if $\int_{1}^{\infty}u \log u\pi(\ddr u)<\infty$. Moreover, under this latter condition, almost surely
\[e^{-\Psi'(0+)t}\hat{X}_t(x)\underset{t\rightarrow \infty}{\longrightarrow} \hat{W}(x), \text{ for all } x\notin J,\]
where $ J:=\{x>0: \hat{W}(x+)>\hat{W}(x)\}$
and $\hat{W}$ is the inverse of a subordinator $W$ with Laplace exponent
\[\kappa: \theta\in [0,\infty)\mapsto \theta e^{-\Psi'(0+)\int_{0}^{\theta}\left(\frac{1}{\Psi'(0+)u}-\frac{1}{\Psi(u)}\right)\ddr u}.\] 
\end{corollary}
\begin{example} Let $\gamma>0$. Consider the subcritical Neveu CSBP whose branching mechanism is defined by $\Psi(u):=\gamma (u+1)\log(u+1)$ for all $u\geq 0$. Note that $\Psi'(0+)=\gamma>0$ and $\int^{\infty}\frac{\ddr u}{\Psi(u)}=\infty$. Solving \eqref{v} yields \[v_t(\lambda)=(\lambda+1)^{e^{-\gamma t}}-1 \underset{t\rightarrow \infty}{\sim} \log(1+\lambda) e^{-\gamma t}.\]  By Corollary \ref{LlogL}, almost surely \[e^{-\gamma t}\hat{X}_t(x)\underset{t\rightarrow \infty}{\longrightarrow} \hat{W}(x) \text{ for all } x\geq 0,\]
where $\hat{W}$ is the inverse of a subordinator $W$ with Laplace exponent
\[\kappa(\theta)=\gamma \log(1+\theta)=\int_{0}^{\infty}(1-e^{-\theta x})\nu(\ddr x),\]
with $\nu(\ddr x):=\gamma\frac{e^{-x}}{x}\ddr x$. The limiting process $\hat{W}$ is therefore an inverse Gamma subordinator.
\end{example}

The process $(\hat{W}^{\lambda}(x),x\geq 0)$ can be interpreted as follows. Define a random equivalence relation $\mathscr{A}$ on $(0,\infty)$ via
\begin{center}
 $x\overset{\mathscr{A}}{\sim} y$ if and only if $\hat{W}^{\lambda}(x)=\hat{W}^{\lambda}(y)$.
\end{center}
This induces a random partition of the set $(0,\infty)$ into intervals of constancy of $\hat{W}^{\lambda}$. A simple application of Lemma \ref{lambda} ensures that this partition does not depend on $\lambda$.  
%One writes $\mathscr{A}=\left((x_{i-1},x_i], i\in I\right)$, the partition ordered in the decreasing order of interval lengths.
By definition, the subintervals of the partition $\mathscr{A}$ are made of individuals whose ancestral lineages have the same asymptotic behaviour. The next proposition states that $\mathscr{A}$ corresponds actually to the families of current individuals having a common ancestor.
\begin{proposition}\label{ancestralpartition} For any $x, y \in (0,\infty)$,  \[x\overset{\mathscr{A}}{\sim} y \text{ if and only if } \hat{X}_t(x)=\hat{X}_{t}(y) \text{ for some } t\geq 0.\]
\end{proposition}
%We shall observe later, see Proposition \ref{C(infty)}, that the subintervals  of the partition $\mathscr{A}$ are not only families whose ancestral lineages have same asymptotic behaviours but those with a same common ancestor. 
%We shall actually see  that $\mathscr{A}$ coincides with the genuine ancestral partition into common ancestors.
Theorem \ref{mainthm} and Proposition \ref{ancestralpartition} complete the results obtained under Grey's condition, in Foucart et al. \cite[Sections 4 and 5.3]{FoMaMa2019}, on the long-term behavior of the ancestral lineages. When $\int^{\infty}\frac{\ddr u}{\Psi(u)}<\infty$ (Grey's condition), the process $(\hat{W}^{\infty}(x),x\geq 0)$ is the left-continuous inverse of a compound Poisson process whose jump law is $\nu_\infty$. Therefore it has piecewise constant sample paths with intervals of constancy of i.i.d lengths with law $\nu_\infty$. The partition $\mathscr{A}$ is thus constituted of i.i.d. families with lengths of law $\nu_\infty$, i.e of the form \[\mathscr{A}=\big((0,x_1],(x_1,x_2], \ldots \big) \text{ a.s.},\] 
where $(x_{i},i\geq 1)$ is a random renewal process with jump law $\nu_{\infty}$. Set $x_0=0$. The ancestral lineage of a given family $(x_{i-1},x_{i}]$, for $i\geq 1$, escapes towards $\infty$ at speed \[t\mapsto \frac{\hat{W}^{\infty}(x_{i})}{v_t(\infty)}.\]  Conditionally on $x_i$, we have that $\hat{W}^{\infty}(x_{i})=\mathbbm{e}_1+\ldots +\mathbbm{e}_i$, where the random variables $\mathbbm{e}_1,\mathbbm{e}_2,...$ are independent and standard exponential.
The following figure provides a symbolic representation of the families, their lineages and the process $\hat{W}_\infty$, under Grey's condition.
\begin{figure}[!ht]
\centering \noindent
\includegraphics[height=.20 \textheight]{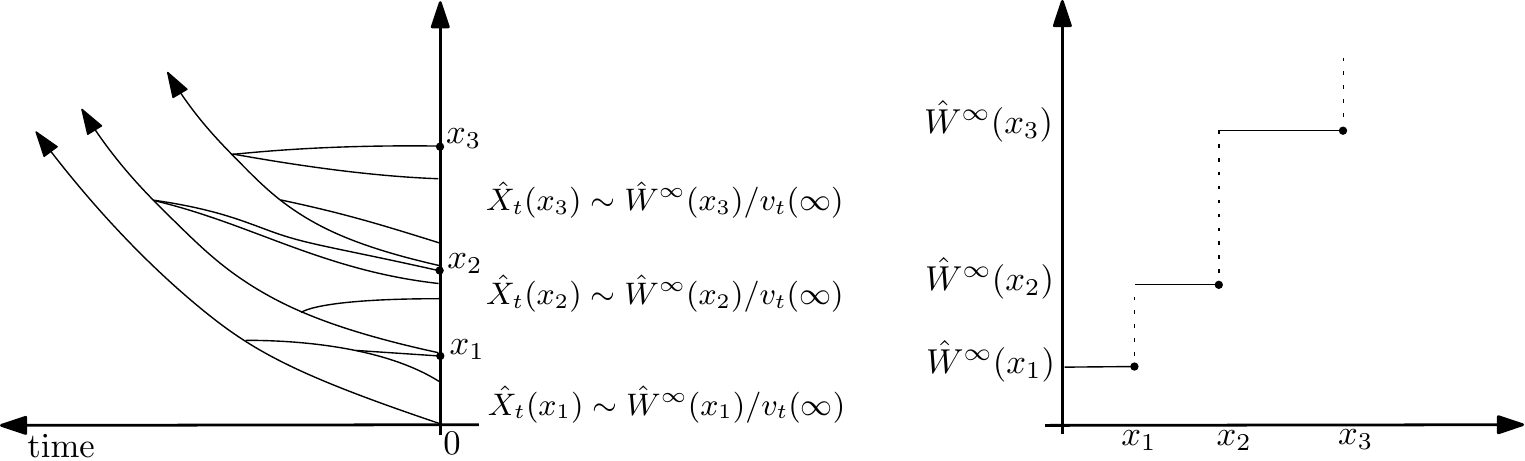}
\caption{Symbolic representation of ancestral families under Grey's condition}
\label{renewal}
\end{figure}

Coalescences between the ancestral lineages inside each family are possibly multiple and can be described using the notion of consecutive coalescents, see \cite[Section 5]{FoMaMa2019}. 

Conditionally given $\mathscr{A}$, for any $j>i$, when $\hat{W}^{\infty}(x_j)-\hat{W}^{\infty}(x_{j-1})=\mathbbm{e}_j>\hat{W}^{\infty}(x_i)-\hat{W}^{\infty}(x_{i-1})=\mathbbm{e}_i$, the ancestor of the family $(x_{j-1},x_j]$ is found \textit{asymptotically earlier} in the past than that of $(x_{i-1},x_i]$. We may thus interpret the exponential random variables $(\mathbbm{e}_i,i\geq 1)$ as measures of ages of the families in $\mathscr{A}$. In the figure above, the divergence towards $\infty$ of the ancestral lineage of $(x_2,x_3]$ is quicker than that of $(x_1,x_2]$, namely $\mathbbm{e}_3>\mathbbm{e}_2$, and
the family $(x_1,x_2]$ is older than $(x_2,x_3]$.

 % the quasi-stationary distribution of the CSBP. %It completes some preliminary observations in law established in \cite[Proposition 2.1]{FoMaMa2019} when Grey's condition holds (namely when families get extinct in finite time almost surely).  

When $\int^{\infty}\frac{\ddr u}{\Psi(u)}=\infty$, the description is more involved since the process $(\hat{W}^{\lambda}(x),x\geq 0)$ has singular continuous paths and any fixed subinterval of $(0,\infty)$ of finite length contains infinitely many small families with positive probability. More precisely, the ancestral families are separated by points $x_i,i\in I$, in the support  $\mathscr{S}$ of the associated singular Stieltjes measure $\ddr \hat{W}^{\lambda}$.
\begin{proposition}\label{dimension} Set $\Psi'(\infty):=\underset{u\rightarrow \infty}{\lim} \frac{\Psi(u)}{u}\in (0,\infty]$. For any $x>0$, the Hausdorff dimension of $\mathscr{S}\cap [0,x]$ is 
 \begin{equation}\label{hausdorff}\mathrm{dim}_{H}(\mathscr{S}\cap [0,x])=\frac{\Psi'(0+)}{\Psi'(\infty)}\in [0,1) \text{ a.s.}
\end{equation} 
\end{proposition}
From \eqref{branchingmechanism}, one sees that $\Psi'(\infty)=\frac{\sigma^2}{2}\cdot \infty+\gamma+\int_{0}^{\infty}x\pi(\ddr x)\in (0,\infty]$. In the case of a CSBP with unbounded variation, namely with $\sigma>0$ or $\int_{0}^{1}x\pi(\ddr x)=\infty$, one has $\Psi'(\infty)=\infty$ and the Hausdorff dimension of $\mathscr{S}$ is zero. In the bounded variation case, \eqref{hausdorff} can be rewritten as
\[\mathrm{dim}_{H}(\mathscr{S}\cap [0,x])=\frac{\gamma}{\gamma+\int_{0}^{\infty}x\pi(\ddr x)} \text{ a.s}.\]
\begin{example} Let $\Psi$ be the branching mechanism with drift $\gamma=1$ and L\'{e}vy measure  $\pi(\ddr x)=x^{-\alpha-1}e^{-x}\ddr x$ with $\alpha\in (0,2)$.
\begin{itemize}
\item[i)] If $\alpha\in (1,2)$, then Grey's condition \eqref{Greycondition} holds, $\Psi'(\infty)=\infty$, and the ancestral families $\mathscr{A}$ are separated along a discrete set $\mathscr{S}$ and $\mathrm{dim}_{H}(\mathscr{S}\cap [0,x])=0$ a.s.
\item[ii)] If $\alpha=1$ (Neveu case), then  Grey's condition \eqref{Greycondition} does not hold, $\Psi'(\infty)=\infty$, and
the ancestral families $\mathscr{A}$ are separated along the set $\mathscr{S}$ and $\mathrm{dim}_{H}(\mathscr{S}\cap [0,x])=0$ a.s.
\item[iii)] If $\alpha<1$, then Grey's condition \eqref{Greycondition} does not hold, $\Psi'(\infty)<\infty$, and the ancestral families $\mathscr{A}$ are separated along the set $\mathscr{S}$ and $\mathrm{dim}_{H}(\mathscr{S}\cap [0,x])=\frac{1}{1+\Gamma(1-\alpha)}$ a.s.
\end{itemize}
\end{example}
In general, inverse subordinators do not have the Markov property. The joint density of the finite-dimensional marginals of  $(\hat{W}^{\lambda}(x),x\geq 0) $ are thus rather involved. We refer to the works of Lageras \cite{lageras2005} and Veillette and Taqqu \cite{VEILLETTE2010697} for information on inverse subordinators. The following proposition is a side result on the one-dimensional laws of the limiting process $(\hat{W}^{\lambda}(x),x\geq 0)$. Recall that $\nu_{\lambda}$ denotes the L\'{e}vy measure of the subordinator $(W^{\lambda}(x),x\geq 0)$. Set $\bar{\nu}_\lambda$ the tail of the measure $\nu_\lambda$: for any $x\geq 0$, $\bar{\nu}_\lambda(x):=\nu_{\lambda}\big((x,\infty)\big)$.
\begin{proposition}\label{densityprop}
The law of $\hat{W}^{\lambda}(x)$ admits the density $g^{\lambda}_x$ defined on $(0,\infty)$ by
\begin{equation}\label{density} g^{\lambda}_x(u):= \int_{0}^{x}\bar{\nu}_\lambda(x-z)\mathbb{P}(W^{\lambda}(u)\in \ddr z).\end{equation}
When $\int^{\infty}\frac{\ddr u}{\Psi(u)}<\infty$, $\hat{W}^{\infty}(x)$ has density $g^{\infty}_x$ :
\begin{equation}\label{densitygrey} g^{\infty}_x(u):= e^{-u}\sum_{n=0}^{\infty}\frac{u^{n}}{n!}\int_{0}^{x}\bar{\nu}_\infty(x-z)\nu_{\infty}^{\star n}(\ddr z).
\end{equation}
\end{proposition}
\section{Proofs}\label{proofs}
We first establish a general result of independent interest relating $\theta$-invariant functions of stochastically monotone processes, with $\theta$-invariant measures of their dual processes, see Section \ref{generalresult}.  We then apply this result in our setting. The asymptotics of a certain $\theta$-invariant function for the dual process $\hat{X}$ is studied. It enables to show an almost sure convergence for the process $(\hat{X}_t(x),t\geq 0)$ started from a fixed value $x$, see Lemma \ref{asconvlambda}. We study the associated limiting process in $x$ and show that it satisfies the properties stated in Theorem \ref{mainthm}, see Lemma \ref{modification}. We will see from the proof of Lemma \ref{asconvlambda} and Lemma \ref{coincide} that Proposition \ref{ancestralpartition} holds true. Proposition \ref{dimension} will be a consequence of Theorem \ref{mainthm}. The proof of Theorem \ref{mainthm} is divided into several lemmas. The first are needed to show the almost sure convergence towards some positive random variable $\hat{W}^{\lambda}(x)$. Finally, we identify the law of the process $(\hat{W}(x),x\geq 0)$.
\subsection{Invariant functions of stochastically monotone Markov processes}\label{generalresult}

In this section, we consider a ``general" standard Markov process $X:=(X_t,t\geq 0)$ with state space $[0,\infty)$, and denote by $(X_t(y),t\geq 0)$ the process started from $y\in [0,\infty)$.
%For the sake of simplicity, we assume the process to be non-explosive. 
Recall that the process  $X$ is said to be stochastically monotone if for any $t\geq 0$ and $x\in [0,\infty)$, the map $y\mapsto \mathbb{P}(X_t(y)\geq x)$ is non-decreasing.
%\[\text{ if } x\leq y \text{ then } \mathbb{P}(X_t(x)\geq z)\leq  \mathbb{P}(X_t(y)\geq z).\]
%Those processes form a broad class of Markov processes.  
Siegmund \cite{MR0431386} has established that if the process $X$ is  stochastically monotone, non-explosive or with boundary $\infty$ absorbing,
and that for any fixed $t$ and $z$, the map $y\mapsto \mathbb{P}(X_t(y)\geq z)$ is right-continuous then there exists a Markov process $\tilde{X}$, the so-called Siegmund dual process, such that for any $t$ and $x,y$
\begin{equation}\label{dualitySiegmund}
\mathbb{P}(X_t(y)\geq x)=\mathbb{P}(\hat{X}_t(x)\leq y).
\end{equation}
%\begin{remark} Notice that if the process $X$ is Feller, then for any $z\in (0,\infty)$, $X_t(x)\underset{x\rightarrow z}{\longrightarrow} X_t(z)$ in law, see e.g. \cite[Lemma 19.3, page 369]{Kallenberg}, and the map $x\mapsto \mathbb{P}(X_t(x)\geq y)$ is right-continuous.
%\end{remark}
The latter identity can be rewritten as
\begin{equation}\label{duality}\mathbb{P}\big(\hat{X}_t(x)>y\big)=\mathbb{P}\big( x>X_t(y)\big).
\end{equation}
%The process $(\hat{X}_t(x),t\geq 0,x\in [0,\infty))$, called Siegmund dual of $(X_t,t\geq 0)$, is a Markov process. 
Our first result shows how to find fundamental martingales for the Siegmund dual process $(\hat{X}_t(x),t\geq 0)$ of any stochastically monotone Markov process $(X_t(x),t\geq 0)$.  Recall $\hat{T}_y$ defined in \eqref{firstpassagetime}. 
\begin{theorem}[Invariant functions of $\hat{X}$]\label{invariantfunction} 
Let $(P_t,t\geq 0)$ be the semigroup of the process $(X_t,t\geq 0)$. Let $\theta\in \mathbb{R}$. If $\mu_{\theta}$ is a positive Borel measure on $(0,\infty)$ satisfying for any  $t\geq 0$, $\mu_{\theta}P_t=e^{\theta t}\mu_{\theta}$, then the functions
$x\mapsto \mu_{\theta}([0,x))$ and $x\mapsto \mu_{\theta}((x,\infty))$, provided they are well-defined, are $\theta$-invariant functions, namely functions $f_\theta$ such that 
%$\bar{\nu}_\theta(x):=\mu_{\theta}\big((x,\infty)\big)<\infty$,
% $f_{\theta}$ defined on $[0,\infty)$ by \[f_{\theta}:y\mapsto \bar{\nu}_{\theta}(y)\]is non-increasing, right-continuous and satisfies: 
 for any $t\geq 0$ and $x\in [0,\infty)$, 
\[\mathbb{E}[f_{\theta}(\hat{X}_t(x))]=e^{\theta t}f_{\theta}(x),\] so that 
\begin{equation}\label{key}\big(e^{-\theta t}f_{\theta}(\hat{X}_t(x)),t\geq 0\big) \text{ is a martingale}.
\end{equation}
In particular, if the process $(\hat{X}_t,t\geq 0)$ has no positive jumps and $\mu_{\theta}$ is finite on $[0,x)$ for all $x>0$, then $f_\theta:x\mapsto \mu_{\theta}([0,x))$ is a well-defined increasing and left-continuous function, and for  all $y\geq x\geq 0$,
\begin{equation}\label{laplacetransformTy}
\mathbb{E}_x[e^{-\theta \hat{T}_y}]=\frac{\mu_{\theta}([0,x))}{\mu_{\theta}([0,y))}.
\end{equation}
\end{theorem}
\begin{remark} 
%\begin{enumerate}
%\item 
A measure $\mu_{\theta}$ verifying $\mu_{\theta}P_t=\theta \mu_{\theta}$ is sometimes called an \textit{eigen-measure} or a $\theta$-invariant measure. If $\theta<0$ and  $\mu_{\theta}$ is a probability measure, then $\mu_{\theta}$ is a quasi-stationary distribution whose rate of decay is $\theta$. Indeed, let $T$ be a killing time, typically  $T=\inf\{t>0: X_t=0\}$, then for any $t>0$, $\mu_{\theta}P_t=\theta \mu_{\theta}$ is equivalent to
$\mathbb{P}_{\mu_{\theta}}(X_t\in \cdot|T>t)=\mu_{\theta}(\cdot)$ and $\mathbb{P}_{\mu_{\theta}}(T>t)=e^{-\theta t}$.
\end{remark}
%\item 
\begin{remark} 
Let $\mathcal{L}$ be the generator of the process $(X_t,t\geq 0)$.  The Kolmogorov forward equation entails that the condition $\mu_{\theta}P_t=e^{\theta t}\mu_{\theta}$ for all $t\geq 0$,  is equivalent to $\mu_{\theta}\mathcal{L}=\theta \mu_{\theta}$ where $\mu_{\theta}\mathcal{L}$ is by definition the measure such that $\langle \mu_{\theta}\mathcal{L},f\rangle=\int \mathcal{L}f(x)\mu_{\theta}(\ddr x)$ for any function $f\in C^{2}_b((0,\infty))$.
\end{remark}

%\end{enumerate}

\begin{proof}
Set $f_\theta(x)=\mu_{\theta}([0,x))$ for all $x>0$. For any $x\in (0,\infty)$ and any $t\geq 0$,
\begin{align*}
\hat{P}_tf_{\theta}(x)=\mathbb{E}\left[f_{\theta}\big(\hat{X}_t(x)\big)\right]&=\mathbb{E}\left[\int \mathbbm{1}_{\{\hat{X}_t(x)>y\}}\mu_{\theta}(\ddr y)\right]\\
&=\mathbb{E}\left[\int \mathbbm{1}_{\{x> X_t(y)\}}\mu_{\theta}(\ddr y)\right] \text{ by the duality relation } \eqref{duality}\\
&=\mu_{\theta}P_t\big([0,x)\big)=e^{\theta t}\mu_{\theta}\big([0,x)\big)=e^{\theta t}f_{\theta}(x).
\end{align*}
The martingale property \eqref{key} follows readily from the Markov property. Note that the map $f_{\theta}:x\mapsto \mu_{\theta}([0,x))$ is left-continuous. We now apply the bounded optional stopping time theorem at time $t\wedge \hat{T}_y$:
\[\mathbb{E}\big[e^{-\theta t\wedge \hat{T}_y}f_{\theta}\big(\hat{X}_{t\wedge \hat{T}_y}(x)\big)\big]=f_{\theta}(x).\]
Since $f_{\theta}$ is non-decreasing and $\hat{X}_{t\wedge \hat{T}_y}(x)\leq y$ a.s, one has for any $t\geq 0$, $f_{\theta}\big(\hat{X}_{t\wedge \hat{T}_y}(x)\big)\leq f_{\theta}(y)$. On the event $\{\hat{T}_y<\infty\}$, the left-continuity of $f_{\theta}$ and the absence of negative jumps in the process $(\hat{X}_t,t\geq 0)$ ensure that $f_{\theta}\big(\hat{X}_{t\wedge \hat{T}_y}(x)\big)\underset{t\rightarrow \infty}{\longrightarrow} f_{\theta}(y)$. This yields 
\[f_{\theta}(x)=\underset{t\rightarrow \infty}{\lim} \mathbb{E}\big[e^{-\theta t\wedge \hat{T}_y}f_{\theta}\big(\hat{X}_{t\wedge \hat{T}_y}(x)\big)\big]=\mathbb{E}\big[e^{-\theta \hat{T}_y}f_{\theta}(y)\mathbbm{1}_{\{\hat{T}_y<\infty\}}\big],\]
which provides the identity \eqref{laplacetransformTy}.
\qed
\end{proof}
\begin{remark} 
Theorem \ref{invariantfunction} holds in general for any stochastically monotone Markov process. In particular, the process is not required to have one-sided jumps. The state space $[0,\infty]$ could also be replaced by a more general nice ordered state space. 
\end{remark}
\begin{remark}
Let $\hat{\mathcal{L}}$ denote the generator of the process $\hat{X}$. An invariant function $f_\theta$ for the semigroup of $\hat{X}$ can be thought as a solution to the equation $\hat{\mathcal{L}}f_\theta=\theta f_{\theta}$. However, when the process has jumps, $\hat{\mathcal{L}}$ is an integro-differential operator and no general theory allows one for identifying solutions of this equation. Lemma \ref{invariantfunction} reveals that for stochastically monotone processes finding a $\theta$-invariant function corresponds to finding a $\theta$-invariant measure for the dual process. This is reminiscent to a result of Cox and R\"osler \cite{MR724061}.
% where it is shown that the functional duality inverts the role of entrance  and exit laws.
%\item Let $\theta>0$. General theory of Markov processes ensures that any one-sided Feller regular process admits a unique increasing left-continuous function $h_{\theta}$ such that for any $x\leq y$
%\[\mathbb{E}_x[e^{-\theta \hat{T}_y}]=\frac{h_{\theta}(x)}{h_{\theta}(y)}.\]
%We refer for instance the reader to Cissé et al. \cite[Proposition 4.1]{} and to \cite{Hawkes}.
%We see therefore that $h_{\theta}$ and $1/f_{\theta}$ coincide. In particular, since $f_{\theta}$ is right-continuous, $h_{\theta}$  is continuous.
\end{remark}

\subsection{Application to CSBPs and their ancestral lineages}
Recall the definition of $(\hat{X}_t(x),t\geq 0)$ as the right-continuous inverse of $(X_{-t,0}(x),t\geq 0)$ and the duality relation \eqref{dualitycsbp}. We start by establishing a second duality relation, which will allow us to apply Theorem \ref{invariantfunction}.
\begin{lemma} For any $t\geq 0$, $x,y\in [0,\infty)$ \begin{equation}\label{duality2}\mathbb{P}\big(\hat{X}_t(x)>y\big)=\mathbb{P}\big( x>X_t(y)\big).
\end{equation}
\end{lemma}
\begin{remark} We stress here on the strict inequalities. Lemma \ref{duality2} entails that  $(\hat{X}_t,t\geq 0)$ is the Siegmund dual, in the sense of Section \ref{generalresult}, of the CSBP $(X_t,t\geq 0)$.
\end{remark}

Recall the action \eqref{generatoronexponential} of the generator $\mathcal{L}$ on exponential functions. We now look for the $\theta$-invariant measures $\mu_{\theta}$ in our setting and  their Laplace transforms explicitly in terms of $\Psi$. The following Lemma holds for general branching mechanism (i.e not necessarily subcritical).
\begin{lemma} 
For any $\theta>0$, the map $c_\theta:q\mapsto e^{-\theta\int_{1}^{q}\frac{\ddr u}{\Psi(u)}}$ is the Laplace transform of a Borel measure $\mu_{\theta}$ on $(0,\infty)$. Moreover, the measure $\mu_\theta$ is $\theta$-invariant for the semigroup $(P_t,t\geq 0)$ of the CSBP $(X_t,t\geq 0)$. 
%Let $f_{\theta}(x)=\mu_{\theta}([0,x[)$ for any $x\geq 0$,  the Laplace transform of $f_{\theta}$ is \[\xi_{\theta}(q):=\int_{0}^{\infty}f_{\theta}(x)e^{-qx}\ddr x=\frac{c_{\theta}(\lambda)}{q}e^{\theta \int_{q}^{\lambda}\frac{\ddr u }{\Psi(u)}}.\] 
\end{lemma}

\begin{proof}
Recall that $\mathcal{L}$ denotes the generator of the CSBP $(X_t(x),t\geq 0)$.
Let $\theta \geq 0$. It is easily checked from the expression of $c_\theta$ that $(-1)^{n}c_{\theta}^{(n)}\geq 0$ on $(0,\infty)$. Bernstein theorem, see e.g. \cite[Theorem 12.b, page 161]{Widder}, guarantees that there exists  a certain Borel measure $\mu_{\theta}$ (possibly infinite) on $[0,\infty)$ such that $c_\theta(q)=\int_{(0,\infty)}e^{-qx}\mu_{\theta}(\ddr x)$ for all $q>0$. We now check that the measure $\mu_{\theta}$ solves $\mu_{\theta}\mathcal{L}=\theta \mu_{\theta}$. Let $q>0$, recall that $\lambda>0$ is fixed. Observe first that $c_\theta$ satisfies the equation  
\begin{equation}\label{ode1}-\Psi(q)c'_\theta(q)=\theta c_\theta(q), \quad c_{\theta}(1)=1.
\end{equation}

%One has
% \[\int_{0}^{\infty}e^{-qx}\mu_{\theta}(\ddr x)=c_{\theta}(\lambda)e^{\theta \int_{q}^{\lambda}\frac{\ddr u}{\Psi(u)}}\]
%for some $c_{\theta}(\lambda)>0$. 
Recall $e_{q}:x\mapsto e^{-qx}$ for any $x,q\geq 0$. Since the linear span of exponential functions is a core for the generator $\mathcal{L}$, it  is enough to verify that 
%$\mu_\theta$ is a $\theta$-invariant measure then necessary
%\begin{equation}\label{LTeq} \langle \mu_{\theta}\mathcal{L},1-e_q\rangle:=\int_{0}^{\infty}\mathcal{L}(1-e_q)(x)\mu_{\theta}(\ddr x)=\langle \theta \mu_{\theta},1-e_q \rangle=\theta \int_{0}^{\infty}(1-e^{-qx})\mu_{\theta}(\ddr x)
%\end{equation}
%One has $\mathcal{L}(1-e_q)(x)=-\Psi(q)xe_q(x)$ and \eqref{LTeq} provides 
%\[-\Psi(q)\int_{0}^{\infty}xe^{-qx}\mu_{\theta}(\ddr x)=\theta \int_{0}^{\infty}(1-e^{-qx})\mu_{\theta}(\ddr x)\]
%and if one sets $c_\theta(q):=\int_{0}^{\infty}(1-e^{-qx})\mu_{\theta}(\ddr x)$, we see that it satisfies the equation  
%\[-\Psi(q)\frac{\partial }{\partial q}c_{\theta}(q)=\theta c_{\theta}(q).\]
%Hence for any fixed $\lambda \in (0,\infty)$ and all $q\in [0,\infty]$
%\[c_{\theta}(q)=c_{\theta}(\lambda)e^{-\theta \int_{\lambda}^{q}\frac{\ddr u}{\Psi(u)}}.\]
%\\
%Similarly,
%If $\mu_\theta$ is a $\theta$-invariant measure then necessary
\begin{equation}\label{LTeq} \langle \mu_{\theta}\mathcal{L},e_q\rangle:=\int_{0}^{\infty}\mathcal{L}e_q(x)\mu_{\theta}(\ddr x)=\langle \theta \mu_{\theta},e_q \rangle=\theta \int_{0}^{\infty}e^{-qx}\mu_{\theta}(\ddr x).
\end{equation}
One has on the other hand $\mathcal{L}e_q(x)=\Psi(q)xe_q(x)$ and \eqref{LTeq} is equivalent to
\begin{equation}\label{ode} \Psi(q)\int_{0}^{\infty}xe^{-qx}\mu_{\theta}(\ddr x)=\theta \int_{0}^{\infty}e^{-qx}\mu_{\theta}(\ddr x),
\end{equation}
%if one sets $d_\theta(q):=\int_{0}^{\infty}e^{-qx}\mu_{\theta}(\ddr x)$, 
which holds true by using \eqref{ode1}.
%for any fixed $\lambda \in (0,\infty)$ and all $q\in [0,\infty]$
%\[d_{\theta}(q)=d_{\theta}(\lambda)e^{-\theta \int_{\lambda}^{q}\frac{\ddr u}{\Psi(u)}}.\]
%Unclear! note also that $\mu_{\theta}$ cannot be a finite measure when $\theta>0$ and we should be careful when using Laplace transforms.
\qed
\end{proof}
\begin{remark} In the subcritical or critical cases, the map $c_\theta$ is completely monotone on $(0,\infty)$ and not defined at $0$, the measure $\mu_{\theta}$ is infinite. In the supercritical case, $c_\theta$ is completely monotone \textit{and} well-defined  and right-continuous at $0$, and $\mu_{\theta}$ is finite.
\end{remark}
According to Theorem \ref{invariantfunction}, the map $f_\theta:x\mapsto \mu_{\theta}([0,x))$ is a $\theta$-invariant function for $(\hat{X}_t,t\geq 0)$. 
The following simple calculation provides an expression of the Laplace transform of $f_{\theta}$. 
%set \[\xi_{\theta}(q):=\int_{0}^{\infty}f_{\theta}(x)e^{-qx}\ddr x=\frac{c_{\theta}(\lambda)}{q}e^{\theta \int_{q}^{1}\frac{\ddr u }{\Psi(u)}}.\] 
%\begin{lemma} 
For any $q>0$,
\begin{equation}\label{LTxi}
\xi_{\theta}(q):=\int_{0}^{\infty}f_{\theta}(y)e^{-qy}\ddr y=\int_{0}^{\infty}\int_{0}^{\infty}\mathbbm{1}_{\{u<y\}}e^{-qy}\mu_{\theta}(\ddr u)\ddr y=\frac{1}{q}c_{\theta}(q).
%\\
%&=\int_{0}^{\infty}\frac{e^{-qu}}{q}\mu_{\theta}(\ddr u)=\frac{1}{q}c_{\theta}(q)
%\\
%&=\frac{d_{\theta}(\lambda)}{q} \exp \left(\theta \int_q^{\lambda}\frac{\ddr u}{\Psi(u)}\right).
\end{equation}
%\end{lemma}
Inverting $\xi_{\theta}$ in order to find $f_{\theta}$
does not seem to be feasible in a general setting, however we shall see in the next lemma that $\xi_{\theta}$ has regular variation properties at $0$, Tauberian theorems will then allow us to find an equivalent at $\infty$ of the function $f_{\theta}$ and hence enable us to investigate more precisely the martingale $(e^{-\theta t}f_{\theta}(\hat{X}_t(x)),t\geq 0)$.

\begin{lemma}\label{regularvar} Assume $\Psi'(0+)\neq 0$. The map $R:q\mapsto e^{-\int_1^{q}\frac{\ddr u}{\Psi(u)}}$ is regularly varying at $0$ with index $-1/\Psi'(0+)$. In particular it takes the form $R(q)=q^{-\frac{1}{\Psi'(0+)}}L_{1}(1/q)$, where $L_{1}$ is a slowly varying function at $\infty$. Moreover, for any $\theta>-\Psi'(0+)$,
\begin{equation}\label{equivf}
f_\theta(y)\underset{y\rightarrow \infty}{\sim} y^{\frac{\theta}{\Psi'(0+)}}\frac{1}{\Gamma\left(1+\frac{\theta}{\Psi'(0+)}\right)}L_{1}(y)^{\theta}=\frac{1}{\Gamma\left(1+\frac{\theta}{\Psi'(0+)}\right)}R(1/y)^{\theta}.
\end{equation}
%Under Grey's condition, we have for all $q>0$, $$R_\infty(q)=e^{\int_{\lambda}^{\infty}\frac{\ddr u}{\Psi(u)}}R_{\lambda}(q).$$
\end{lemma}
\begin{proof} For any $q>0$, 
\begin{align*}
\int_{q}^{1}\frac{\ddr u}{\Psi(u)}&=\int_{q}^{1}\left(\frac{1}{\Psi(u)}-\frac{1}{\Psi'(0+)u}\right) \ddr u+\int_{q}^{1}\frac{\ddr u}{\Psi'(0+)u}\\
&=\int_{1}^{1/q}\left(\frac{1}{u^{2}\Psi(1/u)}-\frac{1}{\Psi'(0+)u}\right) \ddr u.
\end{align*}
Set $\epsilon(u)=\frac{1}{u\Psi(1/u)}-\frac{1}{\Psi'(0+)}$ for $u>0$ and notice that $\epsilon(u)\underset{u\rightarrow \infty}{\longrightarrow} 0$. One has \[R(q)=q^{-\frac{1}{\Psi'(0+)}}\exp\left(\int_{1}^{1/q}\frac{\epsilon(u)}{u}\ddr u\right)\]
and by Karamata's representation theorem, we see that $L_{1}(x):=\exp\left(\int_{1}^{x}\frac{\epsilon(u)}{u}\ddr u\right)$ is slowly varying at $\infty$. Recall $c_\theta(q)=e^{-\theta \int_1^{q}\frac{\ddr u}{\Psi(u)}}$. By \eqref{LTxi}, $\xi_\theta: q\mapsto\frac{c_{\theta}(q)}{q}$ is regularly varying at $0$ with index $\rho=-1-\frac{\theta}{\Psi'(0+)}$. The Tauberian theorem with monotone density, see e.g. \cite[Chapter XIII.5, Theorem 4]{MR0270403}, provides \eqref{equivf}.\qed
\end{proof}

From now on, we focus on the subcritical case $\Psi'(0+)>0$. The next lemma provides an almost sure convergence of the process for fixed $x$ and a representation of the limit $\hat{W}^{\lambda}(x) $
in terms of the first passage time. 
\begin{lemma}\label{asconvlambda} Let $\lambda>0$. For any $x>0$, and $y_0>x$, the following almost sure convergence holds: 
\begin{equation}\label{identity} v_t(\lambda)\hat{X}_t(x)\underset{t\rightarrow \infty}{\longrightarrow} \hat{W}^{\lambda}(x):=c(\lambda,y_0)e^{-\Psi'(0+)\hat{T}^x_{y_0}},
\end{equation}
where $c(\lambda,y_0)$ is a constant independent of $x$. 

Moreover, for any $\lambda>0$ and $x<y<y_0$,
\begin{equation}\label{ratioW}\frac{\hat{X}_t(x)}{\hat{X}_t(y)}\underset{t\rightarrow \infty}{\longrightarrow} \frac{\hat{W}^{\lambda}(x)}{\hat{W}^{\lambda}(y)}=e^{-\Psi'(0+)(\hat{T}^{x}_{y_0}-\hat{T}^{y}_{y_0})}=e^{-\Psi'(0+)\hat{T}^{x}_{y}} \text{ a.s}.
\end{equation}
%Moreover, for any $x$, and $\lambda\neq \lambda'$, $\hat{W}^{\lambda}(x)=c_{\lambda,\lambda'}\hat{W}^{\lambda'}(x)$ almost surely for some constant $c_{\lambda,\lambda'}>0$, and for any $\lambda>0$ and $x<y<y_0$,
%\begin{equation}\label{ratioW}\frac{\hat{X}_t(x)}{\hat{X}_t(y)}\underset{t\rightarrow \infty}{\longrightarrow} \frac{\hat{W}^{\lambda}(x)}{\hat{W}^{\lambda}(y)}=e^{-\Psi'(0+)(\hat{T}^{x}_{y_0}-\hat{T}^{y}_{y_0})}=e^{-\Psi'(0+)\hat{T}^{x}_{y}} \text{ a.s}.
%\end{equation}
%towards some positive random variable $$. 
%We denote its limit by
%\underset{t\rightarrow \infty}{\longrightarrow} 
\end{lemma}
\begin{proof}
We first establish the almost sure convergence towards some random variable $\hat{W}^{\lambda}(x)$. By applying Theorem \ref{invariantfunction}, the process $(e^{-t}f_{1}(\hat{X}_t(x)),t\geq 0)$ is a positive martingale. Therefore, the latter converges almost surely towards some random variable $Z(x)$. Set $\beta(1,\lambda):=\frac{\lambda \xi_{1}(\lambda)}{\Gamma(1+1/\Psi'(0+))}$ and $R_{\lambda}(1/y):=\exp\left(\int_{1/y}^{\lambda}\frac{\ddr u }{\Psi(u)}\right)=\exp\left(\int_{1}^{\lambda}\frac{\ddr u }{\Psi(u)}\right)R(1/y)$. Recall that $\hat{X}_t(x)\underset{t\rightarrow \infty}{\longrightarrow} \infty$ a.s, see Proposition \ref{transience}. By Lemma \ref{regularvar}, 
\[e^{-t}f_{1}\big(\hat{X}_t(x)\big)\underset{t\rightarrow \infty}{\sim} e^{-t}\beta(1,\lambda)R_{\lambda}\big(1/\hat{X}_t(x)\big)\underset{t\rightarrow \infty}{\longrightarrow} Z(x) \text{ a.s.}\]
Hence
\begin{equation}\label{asymptoticequiv}R_{\lambda}\big(1/\hat{X}_t(x)\big)\underset{t\rightarrow \infty}{\sim} \frac{e^{t}Z(x)}{\beta(1,\lambda)}.
\end{equation}
Now, using the fact that $R_\lambda$ is non-decreasing and regularly varying at $0$ with index $-\frac{1}{\Psi'(0+)}$, we see that $R_\lambda^{-1}$ is regularly varying at $\infty$ with index $-\Psi'(0+)$. Taking $R_\lambda^{-1}$ in the asymptotic equivalence \eqref{asymptoticequiv} yields
\[\frac{1}{\hat{X}_t(x)} \underset{t\rightarrow \infty}{\sim} R_{\lambda}^{-1}\left(\frac{e^{t}Z(x)}{\beta(1,\lambda)}\right)\underset{t\rightarrow \infty}{\sim} \left(\frac{Z(x)}{\beta(1,\lambda)}\right)^{-\Psi'(0+)}R_{\lambda}^{-1}(e^{t}).\]
By the definition of $\lambda\mapsto v_t(\lambda)$, see \eqref{v}, one has $R_{\lambda}\big(v_t(\lambda)\big)=e^{t}$. 
This allows us to conclude that
\[v_t(\lambda)\hat{X}_t(x)\underset{t\rightarrow \infty}{\longrightarrow} \hat{W}^{\lambda}(x):=\big(Z(x)/\beta(1,\lambda)\big)^{\Psi'(0+)} \text{ a.s.}\]
%\[\frac{R_{\lambda}\big(1/\hat{X}_t(x)\big)}{R_{\lambda}\big(1/v_t(\lambda)\big)}\underset{t\rightarrow \infty}{\sim} R_{\lambda}\big(1/v_t(\lambda)\hat{X}_t(x)\big)\underset{t\rightarrow \infty}{\longrightarrow} \frac{Z(x)}{\beta(\theta,\lambda)} \text{a.s.}.\]
%Therefore $v_t(\lambda)\hat{X}_t(x)$ converges almost surely towards $Z^{\lambda}(x)$ defined such that $1/Z^{\lambda}(x)=R_{\lambda}^{-1}\big(\frac{Z(x)}{\beta(\theta,\lambda)}\big)$. 
We now explain the relation between $\hat{W}^{\lambda}(x)$ and the first passage time. Note that since the process $(\hat{X}_t(x),t\geq 0)$ is transient, for any $y_0>x$, $\hat{T}^x_{y_{0}}<\infty$ a.s. Let $y_0>y$. It suffices to observe that $Z(x):=\underset{t\rightarrow \infty}{\lim} e^{-t}f_{1}(\hat{X}_t(x))$ coincides with $f_1(y_0)e^{-\hat{T}^x_{y_0}}$. Denote by $(\mathcal{F}_t)_{t\geq 0}$ the natural filtration of $\hat{X}$. Let $(M_t,t\geq 0)$ be the martingale defined by $M_t:=\mathbb{E}[e^{-\hat{T}^x_{y_0}}|\mathcal{F}_t]$ for any $t\geq 0$. Note that $M_t\underset{t\rightarrow \infty}{\longrightarrow} e^{-\hat{T}^x_{y_0}}$ almost surely. Conditional on $\{\hat{T}^x_{y_0}>t\}$, we have that $\hat{T}^x_{y_0}=\hat{T}_{y_0}^{\hat{X}_t(x)}+t$. We get by the strong Markov property
\begin{align*}
M_t&=\mathbb{E}[e^{-\hat{T}^x_{y_0}}|\mathcal{F}_t]=\mathbb{E}_{\hat{X}_t(x)}[e^{-(\hat{T}_{y_0}^{\hat{X}_t(x)}+t)}]=e^{-t}f_{1}(\hat{X}_t(x))/f_{1}(y_0)\underset{t\rightarrow \infty}{\longrightarrow} e^{-\hat{T}^x_{y_0}} \text{ a.s.}
\end{align*}
Hence, $Z(x)=f_1(y_0)e^{-\hat{T}^x_{y_0}}$ a.s, and setting $c(\lambda,y_0):=\big(f_1(y_0)/\beta(1,\lambda)\big)^{-\Psi'0+)}$, we have $\hat{W}^{\lambda}(x)=c(\lambda,y_0)e^{-\hat{T}^x_{y_0}}$ a.s.
%The latter is related to the first passage times $\hat{T}_y^{x}$. 
The relation \eqref{ratioW} is plainly checked using the fact that, since there are no positive jumps, $\hat{T}^{x}_y+\hat{T}^{y}_{y_0}=\hat{T}^{x}_{y_0} \text{ a.s.}$
This achieves the proof of Lemma \ref{asconvlambda}. \qed
\end{proof}
The following lemma is a direct consequence of Lemma \ref{asconvlambda}. 
%In particular, Lemma \ref{lambda} is shown.
%\begin{lemma}\label{ratiolimit}
%For any $\lambda>0$ and $x<y<y_0$,
%\begin{equation}\label{ratioW}\frac{\hat{X}_t(x)}{\hat{X}_t(y)}\underset{t\rightarrow \infty}{\longrightarrow} \frac{\hat{W}^{\lambda}(x)}{\hat{W}^{\lambda}(y)}=e^{-\Psi'(0+)(\hat{T}^{x}_{y_0}-\hat{T}^{y}_{y_0})}=e^{-\Psi'(0+)\hat{T}^{x}_{y}} \text{ a.s}.
%\end{equation}
%\end{lemma}
%\begin{proof}
%The fact that there exists a constant $c_{\lambda,\lambda'}$ such that $\hat{W}^{\lambda}(x)=c_{\lambda,\lambda'}\hat{W}^{\lambda'}(x)$ almost surely is a direct consequence of the identity \eqref{identity} for $\hat{W}^{\lambda}$. Simple calculations show that $c_{\lambda,\lambda'}=e^{\Psi'(0+)\int_{\lambda'}^{\lambda}\frac{\ddr u}{\Psi(u)}}$. The relation \eqref{ratioW} is also directly obtained. 
%\qed
%\end{proof}

\begin{lemma}\label{coincide} For $\lambda>0$ and $x<y$,   $\hat{W}^{\lambda}(x)=\hat{W}^{\lambda}(y)$ almost surely if and only if there exists $T<\infty$ such that $\hat{X}_t(x)=\hat{X}_t(y)$ for all $t\geq T$ a.s.
\end{lemma}
\begin{proof}
According to Lemma \ref{asconvlambda}, if $\hat{W}^{\lambda}(x)=\hat{W}^{\lambda}(y)$, then for any $y_0>x,y$, $\hat{T}^{y}_{y_0}=\hat{T}^{x}_{y_0}$. This entails that at time $T:=\hat{T}^{x}_{y_0}$, $\hat{X}_T(x)=\hat{X}_{T}(y)$ and thus that $\hat{X}_t(x)=\hat{X}_t(y)$ for all $t\geq T$ a.s. The other implication is trivial. \qed
\end{proof}
The almost sure convergence in Lemma \ref{asconvlambda} holds for fixed $x$.   We now  further investigate the limiting process in the variable $x$ and establish a stronger convergence. % $(\hat{W}^{\lambda}(x),x\geq 0)$. 
\begin{lemma}\label{modification} The process $(\hat{W}^{\lambda}(x),x\geq 0)$ admits a non-decreasing left-continuous modification. Moreover, setting $J^{\lambda}:=\{x>0: \hat{W}^{\lambda}(x+)>\hat{W}^{\lambda}(x)\}$, one has almost surely,
\begin{equation}\label{asforall} \forall x\notin J^{\lambda},\ v_t(\lambda)\hat{X}_t(x)\underset{t\rightarrow \infty}{\longrightarrow} \hat{W}^{\lambda}(x).
\end{equation}
\end{lemma}
\begin{proof}
Recall that by Lemma \ref{asconvlambda}, for any fixed $x$,
$v_t(\lambda)\hat{X}_t(x)$ converges almost surely as $t$ goes to $\infty$ towards a random variable denoted by $\hat{W}^{\lambda}(x)$. Consider the almost sure event $\Omega_1$ on which all random variables $(\hat{W}^{\lambda}(q),q\in \mathbb{Q}_+)$ are defined. Since for any $t\geq 0$, the process $(\hat{X}_t(x),x\geq 0)$ is non-decreasing then for any rational numbers $q'\geq q\geq 0$, one has 
$\hat{W}^{\lambda}(q')\geq \hat{W}^{\lambda}(q)$. We work deterministically on  $\Omega_1$ and define for all non-negative real number $x$, \[\tilde{\hat{W}}^{\lambda}(x):=\!\!\underset{q\uparrow x,\atop q\in \mathbb{Q}_+}\lim \! \hat{W}^{\lambda}(q).\] The process $(\tilde{\hat{W}}^{\lambda}(x),x\geq 0)$ is well-defined on $\Omega_1$ and we define it as the null process on $\Omega\setminus \Omega_1$. By construction, $(\tilde{\hat{W}}^{\lambda}(x),x\geq 0)$ is left-continuous. It has right limits since it is non-decreasing. We now show that for any $x\geq 0$, $\mathbb{P}(\hat{W}^{\lambda}(x)=\tilde{\hat{W}}^{\lambda}(x))=1$. By the identity \eqref{ratioW}, on the event $\Omega_1$, we have that for all $q\in \mathbb{Q}_+$ such that $q<x$ 
\begin{equation}\label{WLT}\hat{W}^{\lambda}(q)=\hat{W}^{\lambda}(x)e^{-\Psi'(0+)\hat{T}^{q}_{x}} \text{ a.s}.
\end{equation}
where we recall that $\hat{T}^{q}_{x}:=\inf\{t\geq 0: \hat{X}_t(q)>x\}$. 
We now verify that $\hat{T}^{q}_{x}\underset{q\rightarrow x \atop q<x}{\longrightarrow} 0 $ a.s. Since almost surely for all $q\leq q'$, $\hat{X}_t(q)\leq \hat{X}_t(q')$, then $\hat{T}^{q'}_{x}\leq \hat{T}^{q}_{x}$. Recall the Laplace transform of $\hat{T}^{q}_{x}$ given in Theorem \ref{invariantfunction}, since the function $x\mapsto \mu_{\theta}([0,x))$ is left-continuous, we have that
$\mathbb{E}[e^{-\theta \hat{T}^{q}_{x}}]=\frac{\mu_{\theta}([0,q))}{\mu_{\theta}([0,x))}\underset{q\rightarrow x \atop q<x}{\longrightarrow} 1$; hence $\hat{T}^{q}_{x}\underset{q\rightarrow x \atop q<x}{\longrightarrow} 0 $ a.s. The identity \eqref{WLT}  entails that $\hat{W}^{\lambda}(q)\underset{q\uparrow x, q\in \mathbb{Q}}{\longrightarrow} \hat{W}^{\lambda}(x)$, hence $\hat{W}^{\lambda}(x)=\tilde{\hat{W}}^{\lambda}(x)$ a.s. and $\tilde{\hat{W}}^{\lambda}$ is a left-continuous version of the process $\hat{W}^{\lambda}$. The first statement of the lemma is established.  It remains to see that the almost sure pointwise convergence holds outside the set of jumps $J^{\lambda}$. 
%For all $x$, almost surely $v_t(\lambda)\hat{X}_t(x)\underset{t\rightarrow \infty}{\longrightarrow} \hat{W}^{\lambda}(x).$ 
Almost surely for all $x\notin J^{\lambda}$, one can choose $q'$ and $q$ two rational numbers such that $q'<x<q$. Since $\hat{X}_t(q')\leq\hat{X}_t(x)\leq \hat{X}_t(q)$ for all $t$, one has
\[\tilde{\hat{W}}^{\lambda}(q')\leq \underset{t\rightarrow \infty}{\liminf}\  v_t(\lambda)\hat{X}_t(x)\leq \underset{t\rightarrow \infty}{\limsup}\ v_t(\lambda)\hat{X}_t(x)\leq \tilde{\hat{W}}^{\lambda}(q).\]
Since $\tilde{\hat{W}}^{\lambda}$ has left-continuous paths with right limits and $x\notin J^{\lambda}$, both sides of the inequalities above converge towards the same value $\hat{W}^{\lambda}(x)$ when $q'\uparrow x$ and $q\downarrow x$. This allows us to claim \eqref{asforall}. \qed
\end{proof}
From now on we work with the left-continuous version of $\hat{W}^{\lambda}$ and give up the notation $\tilde{\hat{W}}^{\lambda}$. The next lemma sheds some light on the role of the parameter $\lambda$ and  provides Lemma \ref{lambda}.
%\[v_t(\lambda)\hat{X}_t(x)\underset{t\rightarrow \infty}{\longrightarrow} \hat{W}^{\lambda}(x)\]
%First one sees from Lemma \ref{scalelambda} that for any $\lambda\neq \lambda'$, the processes $(\hat{W}^{\lambda}(x),x\geq 0)$ and $(\hat{W}^{\lambda'}(x),x\geq 0)$ are proportional. 

\begin{lemma}\label{scalelambda} For any $\lambda>0$ and $\lambda'>0$,
\begin{equation}\frac{v_t(\lambda)}{v_t(\lambda')}\underset{t\rightarrow \infty}{\longrightarrow}c_{\lambda,\lambda'}:=e^{\Psi'(0+)\int_{\lambda'}^{\lambda}\frac{\ddr u}{\Psi(u)}}.
\end{equation}
Moreover
$\hat{W}^{\lambda}(x)=c_{\lambda,\lambda'}\hat{W}^{\lambda'}(x)$ for all $x\in (0,\infty)$ almost surely.
\end{lemma}
\begin{proof}
%Recall \eqref{v}, one has $\frac{\ddr v_t(\lambda)}{\ddr t}=-\Psi(v_{t}(\lambda))$, with $v_0(\lambda)=\lambda$, for any $\lambda\geq 0$ and $t\geq 0$. Moreover $v_{t+s}(\lambda)=v_{t}\circ v_{s}(\lambda)$ for any $s,t$. 
Since $\Psi'(0+)\geq 0$, $v_t(\lambda)\underset{t\rightarrow \infty}{\longrightarrow} 0$.  Moreover, $\underset{t\rightarrow \infty}{\lim} \! \uparrow \frac{\Psi(v_t(u))}{v_t(u)}=\Psi'(0+)$. Recall that $\frac{\ddr}{\ddr \lambda}v_t(\lambda)=\frac{\Psi(v_t(\lambda))}{\Psi(\lambda)}$. Therefore for any $\lambda\neq \lambda'$,
\begin{align*}
\frac{v_t(\lambda)}{v_t(\lambda')}&=\exp\left(\int_{\lambda'}^{\lambda}\frac{\ddr }{\ddr u}\log v_t(u) \ddr u\right)=\exp\left(\int_{\lambda'}^{\lambda}\frac{\Psi(v_t(u))}{v_t(u)}\frac{\ddr u}{\Psi(u)}\right)
\end{align*}
and by monotone convergence 
%\[
$\frac{v_t(\lambda)}{v_t(\lambda')}\underset{t\rightarrow \infty}{\longrightarrow} \exp \left(\Psi'(0+)\int_{\lambda'}^{\lambda} \frac{\ddr u}{\Psi(u)}\right)$.
%\]
%\end{align*}
%On the one hand, 
%\begin{align*}
%\int_{v_t(\lambda')}^{v_t(\lambda)}\frac{\ddr u}{\Psi(u)}&\underset{t\rightarrow \infty}{\sim}\frac{1}{\Psi'(0+)}\int_{v_t(\lambda')}^{v_t(\lambda)}\frac{\ddr u}{u}=\frac{1}{\Psi'(0+)}\log\left(\frac{v_t(\lambda)}{v_t(\lambda')}\right).\\
%\end{align*}
%On the other hand,
%\begin{align*}
%\int_{v_t(\lambda')}^{v_t(\lambda)}\frac{\ddr u}{\Psi(u)}&=\int_{v_t(\lambda')}^{\lambda'}\frac{\ddr u}{\Psi(u)}+\int_{\lambda'}^{\lambda}\frac{\ddr u}{\Psi(u)}+\int_{\lambda}^{v_t(\lambda)}\frac{\ddr u}{\Psi(u)}\\
%&=t+\int_{\lambda'}^{\lambda}\frac{\ddr u}{\Psi(u)}-t=\int_{\lambda'}^{\lambda}\frac{\ddr u}{\Psi(u)}.
%\end{align*}
%Therefore, $$\frac{v_t(\lambda)}{v_t(\lambda')}\underset{t\rightarrow \infty}{\longrightarrow} e^{\Psi'(0+)\int_{\lambda'}^{\lambda}\frac{\ddr u}{\Psi(u)}}.$$
We see from Lemma \ref{modification} that $\hat{W}^{\lambda}(x)=c_{\lambda,\lambda'}\hat{W}^{\lambda'}(x)$ for all $x\in (0,\infty)$ almost surely.\qed
\end{proof}

\begin{lemma}\label{levymeasure} 
The map $\kappa_\lambda:q \mapsto \e^{-\Psi'(0+)\int_{q}^{\lambda}\frac{\ddr u}{\Psi(u)}}$ is the Laplace exponent of a drift-free subordinator $W^{\lambda}$. Its L\'{e}vy measure, denoted by $\nu_{\lambda}$, is finite  if and only if $\int^{\infty}\frac{\ddr u}{\Psi(u)}<\infty$. 
\end{lemma}
\begin{proof}
For any $\theta>0$ and any $y \in (0,\infty)$, since $X_{-t,0}(z)$ has the same law as $X_{0,t}(z)$ for all $z$ and $t\geq 0$, one gets by \eqref{semigroupX} and by applying Lemma \ref{scalelambda} \begin{equation}\label{convergenceLE} \mathbb{E}[e^{-\theta X_{-t,0}\left(\frac{y}{v_t(\lambda)}\right)}]=e^{-y\frac{v_t(\theta)}{v_t(\lambda)}}\underset{t\rightarrow \infty}{\longrightarrow} e^{-y \kappa_\lambda(\theta)}.
\end{equation} 
At any time $t$, the process $(X_{-t,0}(y/v_{t}(\lambda)),y\geq 0)$ is a subordinator, the function $\kappa_\lambda$ is therefore the Laplace exponent of a certain subordinator $(W^{\lambda}(y),y\geq 0)$. We show that there is no drift in the subordinator. Recall $\Psi'(\infty):=\underset{u\rightarrow \infty}{\lim} \frac{\Psi(u)}{u}\in (0,\infty]$. Since the case of linear branching mechanism is discarded, i.e $\Psi(q)\not\equiv bq$, one has by convexity, $\Psi'(\infty)>\Psi'(0+)$. Choose $\delta\in (\Psi'(0+),\Psi'(\infty))$. There exists $\lambda_0$ such that for all $u\geq \lambda_0$, $\frac{\Psi(u)}{u}\geq \delta$ and thus $\frac{1}{\Psi(u)}\leq \frac{1}{\delta u}$. Therefore
$$\underset{\theta \rightarrow \infty}{\lim} \int_{\lambda_0}^{\theta}\left(\frac{1}{\Psi'(0+)u}-\frac{1}{\Psi(u)}\right)\ddr u\geq \int_{\lambda_0}^{\infty}\left(\frac{1}{\Psi'(0+)}-\frac{1}{\delta}\right)\frac{\ddr u}{u}=\infty.$$
One deduces that
\begin{align*}
\frac{\kappa_\lambda(\theta)}{\theta}&=\frac{1}{\lambda}\exp\left(-\Psi'(0+)\int_{\lambda}^{\theta}\left(\frac{1}{\Psi'(0+)u}-\frac{1}{\Psi(u)}\right)\ddr u\right)\underset{\theta \rightarrow \infty}{\longrightarrow} 0,
\end{align*}
which entails that there is no drift. Letting $\theta$ to $\infty$ in $\kappa_{\lambda}(\theta)$, we see that 
$$\underset{\theta \rightarrow \infty}{\lim} \kappa_\lambda(\theta)=\nu_{\lambda}((0,\infty))=e^{\Psi'(0+)\int_{\lambda}^{\infty}\frac{\ddr u}{\Psi(u)}},$$
which is finite if and only if $\int^{\infty}\frac{\ddr u}{\Psi(u)}<\infty$. Therefore the L\'{e}vy measure $\nu_\lambda$ is finite if and only if $\int^{\infty}\frac{\ddr u}{\Psi(u)}<\infty$. \qed
\end{proof}
\begin{lemma}\label{lawofhatW}
The process $(\hat{W}^{\lambda}(x),x\geq 0)$ has the same finite-dimensional law as 
\[((W^{\lambda})^{-1}(x),x\geq 0)\] 
where $W^{\lambda}$ is a c\`adl\`ag subordinator  with Laplace exponent $\kappa_\lambda$ and $$(W^{\lambda})^{-1}(x):=\inf \{y\geq 0: W^{\lambda}(y)\geq x\}.$$
Moreover, if $\int^{\infty}\frac{\ddr u}{\Psi(u)}=\infty$, then the process $\hat{W}^{\lambda}$ has continuous paths almost surely.
\end{lemma}
\begin{proof}
By independence and stationarity of the increments of $(X_{-t,0}(y/v_{t}(\lambda)),y\geq 0)$, the convergence in law of the one-dimensional marginal
$(X_{-t,0}(y/v_{t}(\lambda)),t\geq 0)$ as $t$ goes to $\infty$ towards $W^{\lambda}(y)$, established in \eqref{convergenceLE}, entails the convergence of the finite-dimensional marginals.  Since $W^{\lambda}$ has no drift, the range of the subordinator $W^{\lambda}$ contains almost surely no fixed point, see \cite[Proposition 1.9-(i)]{Bertoin1999}. Hence for any $x\in (0,\infty)$ and $y\in (0,\infty)$, $\mathbb{P}(W^{\lambda}(y)=x)=0$. The weak convergence of the finite-dimensional marginals of $(X_{-t,0}(y/v_{t}(\lambda),y\geq 0)$ entails thus that for any $0<y_1<\ldots<y_n$ and $0<x_1<\ldots<x_n$
\begin{align*}
&\underset{t\rightarrow \infty}{\lim}\mathbb{P}\big(X_{-t,0}(y_1/v_t(\lambda))\geq x_1,\ X_{-t,0}(y_2/v_t(\lambda))\geq x_2,\ldots, X_{-t,0}(y_n/v_t(\lambda))\geq x_n \big)\nonumber  \\
&=\mathbb{P}\big(W^{\lambda}(y_1)\geq x_1,\  W^{\lambda}(y_2)\geq x_2,\ \ldots, W^{\lambda}(y_n)\geq x_n\big).
\nonumber 
\end{align*}
%By Theorem \ref{cvthm}, for any $x$ and $y$,\[\mathbb{P}(v_t(\lambda)\hat{X}_t(x)\geq y)\underset{t\rightarrow \infty}{\longrightarrow} \mathbb{P}(\hat{W}^{\lambda}(x)\geq y).\]
By applying Lemma \ref{asconvlambda}, the definition of the flow $(\hat{X}_t(x),x\geq 0)$ and the duality relation \eqref{dualitycsbp3},  one gets the following identities 
\begin{align}
&\mathbb{P}\left(\hat{W}^{\lambda}(x_1)\leq y_1,\ \hat{W}^{\lambda}(x_2)\leq y_1,\ \ldots, \ \hat{W}^{\lambda}(x_n)\leq y_n\right) \\
&=\underset{t\rightarrow \infty}{\lim} \mathbb{P}\left(v_t(\lambda)\hat{X}_t(x_1)\leq y_1, v_t(\lambda)\hat{X}_t(x_2)\leq y_2,\ \ldots, v_t(\lambda)\hat{X}_t(x_n)\leq y_n\right)\nonumber  \\
&=\underset{t\rightarrow \infty}{\lim}\mathbb{P}\big(X_{-t,0}(y_1/v_t(\lambda))\geq x_1,\ X_{-t,0}(y_2/v_t(\lambda))\geq x_2,\ldots, X_{-t,0}(y_n/v_t(\lambda))\geq x_n \big)\nonumber  \\
&=\mathbb{P}\big(W^{\lambda}(y_1)\geq x_1,\  W^{\lambda}(y_2)\geq x_2,\ \ldots, W^{\lambda}(y_n)\geq x_n\big)
\nonumber \\
&=\mathbb{P}\left((W^{\lambda})^{-1}(x_1)\leq y_1,\ \ (W^{\lambda})^{-1}(x_2)\leq y_2,\ \ldots, \ (W^{\lambda})^{-1}(x_n)\leq y_n\right).\nonumber 
\end{align}
The fact that when $\int^{\infty}\frac{1}{\Psi(u)}\ddr u=\infty$, there are no jumps, i.e $J^{\lambda}=\emptyset$ a.s, comes from the fact the processes $(W^{\lambda})^{-1}$ and $\hat{W}^{\lambda}$, where $W^{\lambda}$ is a subordinator with infinite L\'{e}vy measure, have the same law. 
%Indeed, if $J^{\lambda }\neq \emptyset$, then $\mathbb{P}(W^{\lambda}(x)=W^{\lambda}(y) \text{ for some } x<y)>0$. The latter contradicts the fact that the L\'evy measure of 
Since there is no drift in $W^{\lambda}$, the sample paths of $\hat{W}^{\lambda}$ are pure singular continuous functions. \qed
\end{proof}
\begin{remark} Recall the expression of $\hat{W}^{\lambda}(x)$ in terms of $\hat{T}^{x}_{y_0}$ given in Lemma \ref{asconvlambda}.  Lemma \ref{lawofhatW} ensures that the process $W^{\lambda}$ defined at any $y\geq 0$ by  \[W^{\lambda}(y):=\inf\{x>0: \hat{W}^{\lambda}(x)>y\}=\inf\left\{x>0: \hat{T}^{x}_{y_0}<\left(\log c(\lambda,y_0)-\log y\right)/\Psi'(0+)\right\}\]
is a c\`adl\`ag subordinator with Laplace exponent $\kappa_{\lambda}$.
%, is a subordinator is to relate it to the first passage time $(\hat{T}_y,y\geq 0)$ 
\end{remark}
\begin{remark} When $\int^{\infty}\frac{\ddr u}{\Psi(u)}=\infty$,  for any time $t>0$, the subordinator $(X_{-t,0}(x),x\geq 0)$ has an infinite L\'evy measure,  the process $(\hat{X}_t(x),x\geq 0)$ is therefore continuous increasing. Since  its limit $(\hat{W}^{\lambda}(x),x\geq 0)$ is continuous, Dini's theorems ensure that the almost sure convergence  \eqref{asforall} holds true locally uniformly. 
\end{remark}

\noindent \textbf{Proof of Theorem \ref{mainthm}} The main theorem is obtained by combining Lemma \ref{modification}, Lemma \ref{lawofhatW} and Lemma \ref{levymeasure}. \qed

%\begin{lemma}\label{scalelambda} For any $\lambda>0$ and $x>0$,
%\[\frac{v_t(\lambda)}{v_t(\lambda')}\underset{t\rightarrow \infty}{\longrightarrow}e^{\Psi'(0+)\int_{\lambda'}^{\lambda}\frac{\ddr u}{\Psi(u)}}.\]
%\end{lemma}
%\begin{proof}
%On the one hand, 
%\begin{align*}
%\int_{v_t(\lambda')}^{v_t(\lambda)}\frac{\ddr u}{\Psi(u)}&\underset{t\rightarrow \infty}{\sim}\frac{1}{\Psi'(0+)}\int_{v_t(1/x)}^{v_t(\lambda)}\frac{\ddr u}{u}=\frac{1}{\Psi'(0+)}\log\left(\frac{v_t(\lambda)}{v_t(\lambda')}\right).\\
%\end{align*}
%On the other hand,
%\begin{align*}
%\int_{v_t(\lambda')}^{v_t(\lambda)}\frac{\ddr u}{\Psi(u)}&=\int_{v_t(\lambda')}^{\lambda'}\frac{\ddr u}{\Psi(u)}+\int_{\lambda'}^{\lambda}\frac{\ddr u}{\Psi(u)}+\int_{\lambda}^{v_t(\lambda)}\frac{\ddr u}{\Psi(u)}\\
%&=t+\int_{\lambda'}^{\lambda}\frac{\ddr u}{\Psi(u)}-t=\int_{\lambda'}^{\lambda}\frac{\ddr u}{\Psi(u)}.
%\end{align*}
%Therefore, $$\frac{v_t(\lambda)}{v_t(\lambda')}\underset{t\rightarrow \infty}{\longrightarrow} e^{\Psi'(0+)\int_{\lambda'}^{\lambda}\frac{\ddr u}{\Psi(u)}}.$$
%\qed
%\end{proof}

\subsection*{Proof of Corollary \ref{LlogL}}
Recall the statement of Corollary \ref{LlogL}. For any $\lambda>0$, by \eqref{v}, one has for any time $t$,
\[\int_{v_t(\lambda)}^{\lambda}\left(\frac{1}{\Psi'(0+)u}-\frac{1}{\Psi(u)}\right)  \ddr u= \frac{1}{\Psi'(0+)}\log\left(\frac{\lambda}{v_t(\lambda)}\right)-t=\frac{1}{\Psi'(0+)}\log\left(\frac{\lambda}{v_t(\lambda)}e^{-\Psi'(0+)t}\right).\]
Recall that $v_t(\lambda)\underset{t\rightarrow \infty}{\longrightarrow} 0$. Therefore, the asymptotics $v_t(\lambda)\underset{t\rightarrow \infty}{\sim} c_{\lambda}e^{-\Psi'(0+)t}$ holds for some constant $c_\lambda>0$  if and only if \begin{equation}\label{integral}\int_{0}^{\lambda}\left(\frac{1}{\Psi'(0+)u}-\frac{1}{\Psi(u)}\right)  \ddr u <\infty.
\end{equation}
One has, when this is the case,  $c_\lambda=\lambda e^{-\Psi'(0+)\int_{0}^{\lambda}\left(\frac{1}{\Psi'(0+)u}-\frac{1}{\Psi(u)}\right)\ddr u}$.
Since $\Psi(u)\underset{u\rightarrow 0}{\sim} \Psi'(0+)u$, the convergence \eqref{integral}  is equivalent to 
%\[
$\int_{0}^{\lambda}\left(\frac{\Psi(u)-\gamma u}{u^{2}}\right) \ddr u<\infty$,
%\]
where we recall that $\gamma=\Psi'(0+)$ is the linear drift in $\Psi$, see \eqref{branchingmechanism}. Simple calculations from the L\'{e}vy-Khintchine form \eqref{branchingmechanism} ensure that  the latter integral converges if and only if $\int^{\infty}u\log u\pi(\ddr u)<\infty$. We refer for instance to the calculations around Proposition 3.14 in \cite{MR2760602}. By Theorem \ref{mainthm}, almost surely for all $x\notin J^{\lambda}$,
\[e^{-\Psi'(0+)t}\hat{X}_t(x)\underset{t\rightarrow \infty}{\longrightarrow } \frac{1}{c_\lambda}\hat{W}^{\lambda}(x).\]
The process $(\frac{1}{c_\lambda}\hat{W}^{\lambda}(x),x\geq 0)$ is the inverse of the subordinator $(W^{\lambda}(c_\lambda x),x\geq 0)$, whose Laplace exponent is $\theta\mapsto c_\lambda\kappa_\lambda(\theta)$. Recall $\kappa_\lambda$, one easily checks that 
\begin{align*}
c_\lambda\kappa_\lambda(\theta)&=\lambda e^{-\Psi'(0+)\left[\int_{0}^{\lambda}\left(\frac{1}{\Psi'(0+)u}-\frac{1}{\Psi(u)}\right)\ddr u-\int_{\lambda}^{\theta}\frac{\ddr u}{\Psi(u)}\right]}
%&=\lambda e^{-\Psi'(0+)\left[\int_{0}^{q}\left(\frac{1}{\Psi'(0+)u}-\frac{1}{\Psi(u)}\right)\ddr u -\int_{\lambda}^{q}\frac{\ddr u}{\Psi'0+)u}\right]}
=\theta e^{-\Psi'(0+)\int_{0}^{\theta}\left(\frac{1}{\Psi'(0+)u}-\frac{1}{\Psi(u)}\right)\ddr u}.
\end{align*}  
\qed
\subsection*{Proof of Proposition \ref{dimension}}
Recall the statement of Proposition \ref{dimension} where $\mathscr{S}$ denotes the support of the random measure $\ddr \hat{W}^{\lambda}$. By Lemma \ref{scalelambda}, $\hat{W}^{\lambda}=c_{\lambda,\lambda'}\hat{W}^{\lambda'}$ almost surely, this entails that $\mathscr{S}$ does not depend on the parameter $\lambda$.  Moreover, we see by Lemma \ref{lawofhatW} that $\mathscr{S}$ is  the range of a subordinator $W^{\lambda}$ with Laplace exponent $\kappa_\lambda$. Having the Laplace exponent of $W^{\lambda}$ at hand, one can directly apply known results on the geometry of the range of a subordinator. By \cite[Theorem 5.1]{Bertoin1999}, for all $x>0$, almost surely 
\[\underline{\mathrm{dim}}_{H}(\mathscr{S}\cap [0,x])=\underline{\mathrm{ind}}(\kappa_\lambda)\text{ and }\overline{\mathrm{dim}}_{H}(\mathscr{S}\cap [0,x])=\overline{\mathrm{ind}}(\kappa_\lambda)\]
with 
\[\underline{\mathrm{ind}}(\kappa_\lambda):=\underset{q\rightarrow \infty}{\liminf}\frac{\log \kappa_\lambda(q)}{\log q} \text{ and } \overline{\mathrm{ind}}(\kappa_\lambda):=\underset{q\rightarrow \infty}{\limsup}\frac{\log \kappa_\lambda(q)}{\log q}.\]
Recall $\kappa_\lambda$, one gets
\[\underline{\mathrm{ind}}(\kappa_\lambda)=\underset{q\rightarrow \infty}{\liminf} \frac{\Psi'(0+)}{\log q} \int_{\lambda}^{q}\frac{\ddr u}{\Psi(u)}.\]
If $\Psi'(\infty):=\underset{u\rightarrow \infty}{\lim} \frac{\Psi(u)}{u}<\infty$, then we see that $\int_{\lambda}^{q}\frac{\ddr u}{\Psi(u)}\underset{q\rightarrow \infty}{\sim} \frac{1}{\Psi'(\infty)}\log q$ and the result is established. If now $\Psi'(\infty)=\infty$, then for any large $D>0$, for large $q$, 
$$\int_{\lambda}^{q}\frac{\ddr u}{\Psi(u)}\leq C_\lambda+\frac{\log q}{D},$$
for some constant $C_\lambda$. We see therefore that 
$\overline{\mathrm{ind}}(\kappa_\lambda)\leq \frac{\Psi'(0+)}{D}$. Since $D$ is arbitrarily large, one can conclude that $\overline{\mathrm{dim}}_{H}(\mathscr{S}\cap [0,x])=\frac{\Psi'(0+)}{\Psi'(\infty)}=0$ a.s. \qed

We now establish Proposition \ref{densityprop} and look for the density of $\hat{W}^{\lambda}(x)$ for fixed $x$.\\

\noindent \textbf{Proof of Proposition \ref{densityprop}}. Let $\lambda>0$ and $\theta>0$. Recall $\kappa_\lambda$ and denote by $\mathbbm{e}_\theta$, $\mathbbm{e}_q$ two independent exponential random variables with parameter $\theta$ and $q$ respectively. By Lemma \ref{lawofhatW}, we have that $\mathbb{P}(W^{\lambda}(\mathbbm{e}_\theta)\geq \mathbbm{e}_{q})=\mathbb{P}(\mathbbm{e}_{\theta}\geq \hat{W}^{\lambda}(\mathbbm{e}_{q}))$, or equivalently
\[\mathbb{E}[e^{-\theta \hat{W}^{\lambda}(\mathbbm{e}_{q})}]=1-\mathbb{E}[e^{-qW^{\lambda}(\mathbbm{e}_\theta)}]. \]
We deduce that
%\begin{align*}
%q\int_{0}^{\infty}\mathbb{E}[e^{-\theta v_t(\lambda) \hat{X}_t(x)}]e^{-qx}\ddr x=\mathbb{E}[e^{-\theta
% \mathbbm{e}_{v_t(q)}}]&=\frac{v_t(q)}{v_t(q)+\theta}=\frac{1}{1+\frac{\theta}{v_t(q)}}.
%\end{align*}
%Let $\lambda>0$, one has
\begin{align*}\label{doubleLaplacetransform}
\int_{0}^{\infty}\mathbb{E}[e^{-\theta \hat{W}^{\lambda}(x)}]e^{-qx}\ddr x&=\frac{1}{q}\left(1-\mathbb{E}[e^{-\kappa_{\lambda}(q)\mathbbm{e}_\theta}]\right)=\frac{1}{q}\left(1-\frac{\theta}{\kappa_{\lambda}(q)+\theta}\right)=\frac{1}{\theta}\frac{\kappa_\lambda(q)}{q}\frac{\theta}{\kappa_\lambda(q)+\theta}\\
&=\frac{1}{\theta}\frac{\kappa_\lambda(q)}{q}\mathbb{E}[e^{-\kappa_\lambda(q)\mathbbm{e}_\theta}]=\frac{1}{\theta}\int_{0}^{\infty}e^{-qu}\bar{\nu}_\lambda(u)\ddr u \int_{0}^{\infty}e^{-qz}\mathbb{P}(W^{\lambda}(\mathbbm{e}_{\theta})\in \ddr z), \nonumber
\end{align*}
where we have used that \[\mathbb{E}[e^{-qW^{\lambda}(\mathbbm{e}_\theta)}]=\mathbb{E}[e^{-\kappa_\lambda(q) \mathbbm{e}_\theta}]
\text{ and } \frac{\kappa_\lambda(q)}{q}=\int_{0}^{\infty}e^{-qu}\bar{\nu}_\lambda(u)\ddr u.\]
By the change of variable $x=u+z$, we obtain 
\begin{align*}\label{limitdoubleLaplacetransform2Nogrey}
\int_{0}^{\infty}\mathbb{E}[e^{-\theta \hat{W}^{\lambda}(x)}]e^{-qx}\ddr x
&=\frac{1}{\theta}\int_{0}^{\infty}e^{-qx}\int_{0}^{\infty}\bar{\nu}_\lambda(x-z)\mathbb{P}(W^{\lambda}(\mathbbm{e}_\lambda)\in \ddr z)
\end{align*}
and deduce that 
%$v_t(\lambda)\hat{X}_t(x)$ converges in law as $t$ goes to $\infty$  towards a random variable $\hat{W}^{\lambda}(x)$ whose Laplace transform is 
\begin{align*} \mathbb{E}[e^{-\theta \hat{W}^{\lambda}(x)}]&=\frac{1}{\theta}\int_{0}^{x}\bar{\nu}_\lambda(x-z)\mathbb{P}(W^{\lambda}(\mathbbm{e}_\theta)\in \ddr z)\\
&=\int_{0}^{\infty}e^{-\theta u} \int_{0}^{x}\bar{\nu}_\lambda (x-z)\mathbb{P}(W^{\lambda}(u)\in \ddr z) \ddr u.
\end{align*}
Thus, the density of $\hat{W}^{\lambda}(x)$ is
%\begin{equation*} 
$g^{\lambda}_x(u):= \int_{0}^{x}\bar{\nu}_\lambda(x-z)\mathbb{P}(W^{\lambda}(u)\in \ddr z).$
%\end{equation*} 
In the case $\int^{\infty}\frac{\ddr u}{\Psi(u)}<\infty$, $(W^{\infty}(u),u\geq 0)$ is a compound Poisson process with intensity $\nu_\infty$, and since $\bar{\nu}_\infty(0)=1$, the formula \eqref{densitygrey} can be plainly checked.\qed
%
%We have not addressed here critical CSBPs. In this case, the process $\hat{X}$ can be either transient or null recurrent. The function $R$ defined in Lemma \ref{regularvar} is not regularly varying at $\infty$ in general but \textit{rapidly varying}. almost sure asymptotic results are therefore more involved and are left for possible future works. We refer the reader to Pakes \cite[Theorem 9, (a)]{pakes2017} for a limit in law in the case of critical discrete branching processes. 

\noindent \textbf{Acknowledgements:} C.F's research is partially  supported by LABEX MME-DII (ANR11-LBX-0023-01). 
%M.M.'s research is supported by {\red{to complete}}.

%\bibliographystyle{amsalpha}
%\bibliography{bibli}

\end{document}